\documentclass[12pt,english]{article}
\usepackage[T1]{fontenc}
\usepackage[latin9]{inputenc}
\usepackage[active]{srcltx}
\usepackage{babel}
\usepackage{mathrsfs}
\usepackage{mathtools}
\usepackage{amsmath}
\usepackage{amssymb}
\usepackage{stmaryrd}
\usepackage{wasysym}
\usepackage[bookmarks=false,
 breaklinks=false,pdfborder={0 0 1},backref=section,colorlinks=false]
 {hyperref}

\makeatletter
\usepackage{color}
\usepackage{graphicx}
\usepackage{amsfonts}
\usepackage{a4wide}
\usepackage{babel}
\usepackage{mathrsfs}

\providecommand{\U}[1]{\protect\rule{.1in}{.1in}}

\newtheorem{theorem}{Theorem}\newtheorem{corollary}[theorem]{Corollary}\newtheorem{lemma}[theorem]{Lemma}\newtheorem{proposition}[theorem]{Proposition}\newtheorem{remark}[theorem]{Remark}\newenvironment{proof}[1][Proof]{\noindent\textbf{#1.} }{\ \rule{0.5em}{0.5em}}

\usepackage[usenames,dvipsnames]{pstricks}
\usepackage{epsfig}
\usepackage{pst-grad}
\usepackage{pst-plot}

\usepackage{tikz}
\usepackage{tikz-cd}

\makeatother

\begin{document}
\title{The maximal monotonicity of the convex subdifferential via upper envelopes }
\author{M.D. Voisei}
\date{{}}

\maketitle
 
\begin{abstract}
We establish equivalent conditions for the convex subdifferential
to be maximal monotone in general settings, expressed in terms of
the associated upper envelopes. 
\end{abstract}

\section{Preliminaries and notations}

The purpose of this paper is to derive equivalent conditions for the
convex subdifferential of every proper lower semicontinuous convex
function to be maximal monotone. To this end, we introduce several
types of upper envelopes associated with a function and investigate
their properties in connection with the maximal monotonicity of the
corresponding subdifferentials.

It is worth mentioning that, for individual functions, the corresponding
question for maximal cyclic monotonicity has already been answered
in \cite[Theorem 36]{gm}.

The main motivation for this paper is to extend the results of \cite{Voisei2019}
from sets to functions. The core of our approach is our characterization
of maximal monotonicity (see \cite[Theorem\ 2.3]{MR2207807}, \cite[Theorem\ 3.4]{MR2453098},
or Theorem \ref{max-sd} below), according to which an operator is
maximal monotone if and only if it is representable and of type NI.
This characterization is expressed in terms of the \emph{Fitzpatrick
function} $\varphi_{T}:X\times X^{*}\rightarrow\overline{\mathbb{R}}$
associated with a multivalued operator $T:X\rightrightarrows X^{*}$,
defined by (see \cite{MR1009594}) 
\begin{equation}
\varphi_{T}(x,x^{*}):=\sup\{\langle x-a,a^{*}\rangle+\langle a,x^{*}\rangle\mid(a,a^{*})\in\operatorname*{Graph}T\},\ (x,x^{*})\in X\times X^{*}.\label{ff}
\end{equation}
Throughout, $(X,\tau)$ is a nontrivial (that is, $X\neq\{0\}$) real
locally convex Hausdorff space, $X^{\ast}$ is its topological dual,
usually endowed with the weak-star topology $w^{*}$, and $(X^{*},w^{*})^{*}$
is identified with $X$. The duality pairing is denoted by $\left\langle x,x^{\ast}\right\rangle :=x^{\ast}(x)=:c(x,x^{\ast})$,
for $x\in X$, $x^{\ast}\in X^{\ast}$ and $\operatorname*{Graph}T=\{(x,x^{*})\in X\times X^{*}\mid x^{*}\in T(x)\}$
stands for the \emph{graph} of $T:X\rightrightarrows X^{*}$.

When $X$ is a Banach space, Rockafellar showed in \cite[Theorem\ A]{MR0262827}
that, for every proper convex lower semicontinuous function $f:X\to\overline{\mathbb{R}}$,
the \emph{convex subdifferential} $\partial f:X\rightrightarrows X^{*}$,
defined by 
\[
x^{*}\in\partial f(x)\ \Leftrightarrow\ f(x)\in\mathbb{R}\ \mathrm{and}\ \forall y\in X,\ f(y)\ge f(x)+\langle y-x,x^{*}\rangle,
\]
is maximal monotone; in short $\partial f\in\mathfrak{M}(X)$. In
particular the \emph{normal cone} to $C$, given by $N_{C}=\partial\iota_{C}$,
is maximal monotone provided that $\emptyset\neq C\subset X$ is closed
convex. Here $\iota_{C}(x)=0$, for $x\in C$; $\iota_{C}(x)=+\infty$,
for $x\in X\setminus C$ denotes the \emph{indicator function} of
$C$.

By the biconjugation formula, and because maximal monotonicity is
a duality property, the same result holds when $X$ is the topological
dual of a Banach space equipped with the weak-star topology. We also
note that, if $X$ is a nonreflexive Banach space, then its dual $X^{*}$
is neither barreled nor Banach under any locally convex topology compatible
with the duality $(X^{*},X)$, in particular under its Mackey topology.

Part of the argument in \cite{MR0262827}, drawing on \cite[Proposition 6 (b), (c)]{MR0166579}
and \cite[(A), (B)]{MR0178103} is to show that the subdifferentiability
domain of a proper lower semicontinuous convex function determines
the function completely, in the sense that the function coincides
with the lower semicontinuous (convex) hull of its restriction to
its subdifferentiability domain. As a consequence, every proper lower
semicontinuous convex function on a Banach space is the supremum of
the supporting affine functions associated with its subgradients (see
again \cite[Theorem\ 2,\ p.\ 609]{MR0178103}).

It is therefore natural to ask whether these properties are characteristic
of the two classes of spaces mentioned above, namely Banach spaces
and duals of Banach spaces endowed with the weak-star topology, and
whether they admit extensions to general locally convex spaces. In
particular, one may look for conditions analogous to conditions (A)
and (B) in \cite{MR0178103} under which every convex subdifferential
is maximal monotone.

It is worth noting that the answer to the question raised in \cite[p.\ 606]{MR0178103}
is negative: ``It would be interesting to know whether every space
for which (A) and (B) always hold is necessarily a Banach space \textquotedbl in
disguise\textquotedbl{} (i.e. in its Mackey topology) as the counterexamples
seem to suggest.''

Indeed, if $X$ is a nonreflexive Banach space, then its topological
dual $X^{*}$ satisfies properties (A) and (B) with respect to the
duality $(X^{*},X)$, while there is no barreled topology on $X^{*}$,
and hence no Banach space topology on $X^{*}$, compatible with $(X^{*},X)$.

Our main argument is based on the introduction of several envelopes
naturally associated with a function whose subdifferential is to be
shown to be maximal monotone (see Section 3 below).

\strut

As usual, given a locally convex space $(E,\mu)$ and $A\subset E$
we denote by 
\begin{itemize}
\item ``$A^{\#}$'' the \emph{portable hull} of $A$ , that is, the intersection
of all supporting half-spaces that contain A and are supported at
points of $A$ (see \cite{Voisei2019} for details); 
\item ``$\operatorname*{conv}A$'' the \emph{convex hull} of $A$, defined
by $x\in\operatorname*{conv}A$ if $(x,1)={\displaystyle \sum_{k=1}^{n}}\lambda_{k}(x_{k},1)$
for some $(\lambda_{k})_{k\in\overline{1,n}}\subset\mathbb{R}^{+}:=[0,+\infty)$,
$(x_{k})_{k\in\overline{1,n}}\subset A$; 
\item ``$\operatorname*{cl}_{\mu}(A)=\overline{A}^{\mu}$'' the $\mu-$\emph{closure}
of $A$, defined by $x\in\operatorname*{cl}_{\mu}A$ if $x_{i}\stackrel{\mu}{\to}x$
for some net $(x_{i})_{i}\subset A$, where ``$\stackrel{\mu}{\to}$''
denotes convergence with respect to the topology $\mu$; 
\item ``$\operatorname*{int}_{\mu}A$'' the $\mu-$\emph{topological interior
}of $A$, defined by $x\in\operatorname*{int}_{\mu}A$ if there exists
$V\in\mathscr{V}_{\mu}(x)$ such that $V\subset A$, where $\mathscr{V}_{\mu}(x)$
denotes the \emph{system of $\mu-$neighborhoods of} $x\in E$. The
family of $\mu-$closed convex neighborhoods of $x\in E$ is denoted
by $\mathscr{V}_{\mu}^{c}(x)$. 
\item ``$\operatorname*{core}A$'' the \emph{algebraic interior} of $A$,
defined by $x\in\operatorname*{core}A$ if for every $e\in E$ there
is $\rho>0$ such that, for every $s\in[0,\rho]$, $se+(1-s)x\in A$. 
\end{itemize}
For $f,g:E\rightarrow\overline{\mathbb{R}}$ we set $[f\leq g]:=\{x\in E\mid f(x)\leq g(x)\}$,
and define the sets $[f=g]$, $[f<g]$, and $[f>g]$ similarly. We
write $f\ge g$ whenever $f(z)\ge g(z)$, for every $z\in E$.

For a multifunction $T:X\rightrightarrows X^{*}$, we denote by 
\[
D(T)=\operatorname*{Pr}\!_{X}(\operatorname*{Graph}T),\ R(T)=\operatorname*{Pr}\!_{X^{*}}(\operatorname*{Graph}T)
\]
the domain and the range of $T$, respectively, where $\operatorname*{Pr}_{X}$,
$\operatorname*{Pr}_{X^{*}}$ denote the canonical projections of
$X\times X^{*}$ onto $X$ and $X^{*}$. Whenever no confusion is
likely to arise, we identify the multifunction $T:X\rightrightarrows X^{\ast}$
with its graph $\operatorname*{Graph}T\subset X\times X^{*}$.

We consider the following classes of functions and operators on $(X,\tau)$ 
\begin{description}
\item [{$\Lambda(X)$,}] the class of proper convex functions $f:X\rightarrow\overline{\mathbb{R}}$.
Recall that $f$ is \emph{proper} if $\operatorname*{dom}f:=\{x\in X\mid f(x)<\infty\}$
is nonempty and $f$ never takes the value $-\infty$; 
\item [{$\Gamma_{\tau}(X)$,}] the class of functions $f\in\Lambda(X)$
that are $\tau$--lower semicontinuous; 
\item [{$\mathcal{M}(X)$,}] the class of nonempty monotone operators $T:X\rightrightarrows X^{\ast}$
(that is, $\operatorname*{Graph}T\neq\emptyset$). Recall that $T:X\rightrightarrows X^{\ast}$
is \emph{monotone} if, for all $(x_{1},x_{1}^{*}),(x_{2},x_{2}^{*})\in\operatorname*{Graph}T$,
$\left\langle x_{1}-x_{2},x_{1}^{\ast}-x_{2}^{\ast}\right\rangle \geq0$; 
\item [{$\mathfrak{M}(X)$,}] the class of maximal monotone operators $T:X\rightrightarrows X^{*}$,
where maximality is understood with respect to graph inclusion in
$X\times X^{*}$. 
\end{description}
To a proper function $f:(X,\tau)\rightarrow\overline{\mathbb{R}}$
we associate the following notions: 
\begin{description}
\item [{$\operatorname*{Epi}f:=\{(x,t)\in X\times\mathbb{R}\mid f(x)\leq t\}$,}] the
\emph{epigraph} of $f$; 
\item [{$\operatorname*{conv}f:X\rightarrow\overline{\mathbb{R}}$,}] the
\emph{convex hull} of $f$, that is, the greatest convex minorant
of $f$, defined by $(\operatorname*{conv}f)(x):=\inf\{t\in\mathbb{R}\mid(x,t)\in\operatorname*{conv}(\operatorname*{Epi}f)\}$,
$x\in X$; 
\item [{$\operatorname*{cl}_{\tau}\operatorname*{conv}f:X\rightarrow\overline{\mathbb{R}}$,}] the
\emph{$\tau-$}lower semicontinuous\emph{ convex hull} of $f$, that
is, the greatest \emph{$\tau$--}lower semicontinuous convex minorant
of $f$, defined by $(\operatorname*{cl}_{\tau}\operatorname*{conv}f)(x):=\inf\{t\in\mathbb{R}\mid(x,t)\in\operatorname*{cl}_{\tau\times|\cdot|}(\operatorname*{conv}\operatorname*{Epi}f)\}$,
$x\in X$; 
\item [{$f^{\ast}:X^{\ast}\rightarrow\overline{\mathbb{R}}$,}] the \emph{convex
conjugate} of $f:X\rightarrow\overline{\mathbb{R}}$ with respect
to the dual system $(X,X^{\ast})$, defined by $f^{\ast}(x^{\ast}):=\sup\{\left\langle x,x^{\ast}\right\rangle -f(x)\mid x\in X\}$,
$x^{\ast}\in X^{\ast}$. 
\end{description}
Accordingly, for $C\subset X$, the support function of $C$ is defined
by $\sigma_{C}(x^{*}):=\sup\{\langle x,x^{*}\rangle\mid x\in C\}=\iota_{C}^{*}(x^{*})$,
$x^{*}\in X^{*}$. Recall that $f^{**}:=(f^{*})^{*}=\operatorname*{cl}_{\tau}\operatorname*{conv}f$
whenever $\operatorname*{cl}_{\tau}\operatorname*{conv}f$, or equivalently
$f^{*}$, is proper. Here, for functions defined on $X^{*}$, all
conjugates are taken with respect to the dual system $(X^{*},X)$.

Throughout this article we adopt the conventions $\infty-\infty=\infty$,
$\sup\emptyset=-\infty$, $\inf\emptyset=\infty$, and we suppress
the topology in the notation whenever no ambiguity can arise.

\section{Functions with maximal monotone subdifferentials}

In this section, we study properties of functions $f:X\to\overline{\mathbb{R}}$
for which $\partial f\in\mathfrak{M}(X)$, and we seek to identify
the natural framework in which one can formulate conditions ensuring
that the convex subdifferential is maximal monotone.

\begin{theorem} \label{dfdom} Let $X$ be a locally convex space
and let $f:X\to\overline{\mathbb{R}}$. Then, for every $(x,x^{*})\in X\times X^{*}$,
\begin{equation}
\varphi_{\partial f}(x,x^{*})\le f^{**}(x)+f^{*}(x^{*}).\label{ineq-fidf}
\end{equation}
\emph{(i)} If $\operatorname*{Graph}(\partial f)\neq\emptyset$ then
$f$, $\operatorname*{cl}\operatorname*{conv}f$, $f^{*}$ are proper,
$f^{**}=\operatorname*{cl}\operatorname*{conv}f$, 
\begin{equation}
\operatorname*{Graph}(\partial f)=\operatorname*{Graph}(\partial(\operatorname*{cl}\operatorname*{conv}f))\cap(C_{f}\times X^{*}),\label{eq:-13}
\end{equation}
where $C_{f}\coloneqq\{x\in X\mid f(x)=(\operatorname*{cl}\operatorname*{conv}f)(x)\}$.
In particular, for every $x\in D(\partial f)$, $f(x)=(\operatorname*{cl}\operatorname*{conv}f)(x)$
and $\partial f(x)=\partial(\operatorname*{cl}\operatorname*{conv}f)(x)$.

If, in addition, $\partial f\in\mathfrak{M}(X)$ then 
\[
\partial f=\partial(\operatorname*{cl}\operatorname*{conv}f)\ \mathrm{and}\ \operatorname*{cl}(\operatorname*{dom}f)=\operatorname*{cl}(\operatorname*{dom}(\operatorname*{cl}\operatorname*{conv}f))=\operatorname*{cl}\operatorname*{conv}D(\partial f).
\]

\medskip{}

\noindent\emph{(ii)} If $h:\mathbb{R}\to\overline{\mathbb{R}}$ satisfies
$\partial h\in\mathfrak{M}(\mathbb{R})$, then $h$ is proper and
convex. There exists a proper convex $g:\mathbb{R}\to\overline{\mathbb{R}}$
which is not lower semicontinuous and still satisfies $\partial g\in\mathfrak{M}(\mathbb{R})$.

\medskip{}

\noindent\emph{(iii)} If $X$ has dimension at least two, then there
exists $g:X\to\overline{\mathbb{R}}$ which is neither convex nor
lower semicontinuous and such that $\partial g\in\mathfrak{M}(X)$.

\medskip{}

\noindent\emph{(iv)} Assume, in addition, that $X$ is a Banach space
or that $f$ is continuous at some $\overline{x}\in\operatorname*{dom}f$.
Then 
\begin{center}
$f$ is lower semicontinuous and $\partial f\in\mathfrak{M}(X)$ if
and only if $f\in\Gamma(X)$. 
\par\end{center}

\medskip{}

\noindent\emph{(v)} If $\partial(\operatorname*{cl}\operatorname*{conv}f)\in\mathfrak{M}(X)$,
then $f$ is proper. Nevertheless, there exists $g:X\to\overline{\mathbb{R}}$
which is neither convex nor lower semicontinuous, for which $\partial(\operatorname*{cl}\operatorname*{conv}g)\in\mathfrak{M}(X)$
and $\operatorname*{Graph}(\partial g)\subsetneq\operatorname*{Graph}(\partial(\operatorname*{cl}\operatorname*{conv}g))$.
\end{theorem}

\begin{proof} For every $(x,x^{*})\in X\times X^{*}$ 
\[
\begin{aligned}\varphi_{\partial f}(x,x^{*}) & =\sup\{\langle x-a,a^{*}\rangle+\langle a,x^{*}\rangle\mid(a,a^{*})\in\operatorname*{Graph}(\partial f)\}\\
 & =\sup\{\langle x,a^{*}\rangle-f^{*}(a^{*})+\langle a,x^{*}\rangle-f(a)\mid(a,a^{*})\in\operatorname*{Graph}(\partial f)\}\\
 & \le\sup\{\langle x,a^{*}\rangle-f^{*}(a^{*})\mid a^{*}\in X^{*}\}+\sup\{\langle a,x^{*}\rangle-f(a)\mid a\in X\}\\
 & =f^{**}(x)+f^{*}(x^{*}).
\end{aligned}
\]

(i) Let $\bar{x}^{*}\in\partial f(\bar{x})$, i.e., $f(\bar{x})$
is finite and, for every $x\in X$, 
\begin{equation}
f(x)\ge h(x):=f(\bar{x})+\langle x-\bar{x},\bar{x}^{*}\rangle.\label{eq:-2}
\end{equation}

Note that $h(\bar{x})=f(\bar{x})$ and $h\in\Gamma(X)$. From $h\le f$
we get $h\le\operatorname*{cl}\operatorname*{conv}f\le f$. In particular
$(\operatorname*{cl}\operatorname*{conv}f)(\bar{x})=f(\bar{x})$ and
$\bar{x}^{*}\in\partial(\operatorname*{cl}\operatorname*{conv}f)(\bar{x})$.
The direct inclusion in (\ref{eq:-13}) is proved.

From (\ref{eq:-2}) and since $f(\bar{x})$ is finite, $f$, $\operatorname*{cl}\operatorname*{conv}f$
are proper followed by $f^{*}=(\operatorname*{cl}\operatorname*{conv}f)^{*}\in\Gamma(X^{*})$
and $f^{**}=\operatorname*{cl}\operatorname*{conv}f$.

Conversely, assume that $\bar{x}^{*}\in\partial(\operatorname*{cl}\operatorname*{conv}f)(\bar{x})$
and $(\operatorname*{cl}\operatorname*{conv}f)(\bar{x})=f(\bar{x})$.
Since $\operatorname*{cl}\operatorname*{conv}f\le f$, it follows
that $\bar{x}^{*}\in\partial f(\bar{x})$, thereby establishing (\ref{eq:-13}).

If, in addition, $\partial f\in\mathfrak{M}(X)$ then (\ref{eq:-13})
provides $\partial f=\partial(\operatorname*{cl}\operatorname*{conv}f)$.
From (\ref{ineq-fidf}) and \cite[Lemma\ 3,\ p.\ 279]{MR3492118}
or \cite[Proposition 16 (24), p. 1316]{MR3825440} we get 
\begin{equation}
D(\partial f)\subset\operatorname*{dom}f\subset\operatorname*{dom}(\operatorname*{cl}\operatorname*{conv}f)=\operatorname*{dom}f^{**}\subset\Pr\nolimits_{X}(\operatorname*{dom}\varphi_{\partial f})\subset\operatorname*{cl}\operatorname*{conv}D(\partial f);\label{eq:-3}
\end{equation}
whence $\operatorname*{cl}(\operatorname*{dom}f)=\operatorname*{cl}(\operatorname*{dom}(\operatorname*{cl}\operatorname*{conv}f))=\operatorname*{cl}\operatorname*{conv}D(\partial f)$.

(ii) For a maximal monotone operator on a finite-dimensional space,
the closure of its domain is convex, and the relative topological
interior of its domain is nonempty, convex, and dense in the domain
(see e.g. \cite[p.\ 135]{MR0132379}, \cite[Theorems 6.2, 6.3]{MR1451876},
\cite[Theorem 12.41, p. 554]{MR1491362}). Hence $D(\partial h)$
is a possibly degenerate interval of $\mathbb{R}$. By (i), this interval
is contained and dense in $\operatorname*{dom}h$. Therefore $\operatorname*{dom}h$
is an interval with $\operatorname*{cl}\operatorname*{dom}h=\operatorname*{cl}D(\partial h)$
and $\operatorname*{int}\operatorname*{dom}h=\operatorname*{int}D(\partial h)$.

If $D(\partial h)=\{x_{0}\}$, then $\partial h\in\mathfrak{M}(\mathbb{R})$
implies that $\operatorname*{Graph}(\partial h)=\{x_{0}\}\times\mathbb{R}$.
Consequently, since $h$ is proper, $\operatorname*{dom}h=\{x_{0}\}$
and so $h=h(x_{0})+\iota_{\{x_{0}\}}\in\Gamma(\mathbb{R})$.

If $D(\partial h)$ is a non-degenerate interval, then for every $x,y\in\operatorname*{dom}h$
with $x\neq y$, and every $t\in(0,1)$, 
\[
x_{t}:=tx+(1-t)y\in\operatorname*{int}(\operatorname*{dom}h)\subset D(\partial h).
\]
Choose $z\in\partial h(x_{t})$. Then 
\[
z(x-x_{t})\le h(x)-h(x_{t})\ \mathrm{and}\ z(y-x_{t})\le h(y)-h(x_{t}).
\]
Multiply the first inequality by $t$, the second by $(1-t)$, and
then adding them, we get $h(x_{t})\le th(x)+(1-t)h(y)$. Hence $h$
is convex.

Let $\bar{g}:\mathbb{R}\to\overline{\mathbb{R}}$ be given by $\bar{g}(x):=\infty$
for $x<0$; $\bar{g}(x):=-\sqrt{x}$ for $x\geq0$.

Then $\overline{g}\in\Gamma(\mathbb{R})$, and therefore $\partial\overline{g}\in\mathfrak{M}(\mathbb{R})$.
Take $g:\mathbb{R}\to\overline{\mathbb{R}}$ defined by $g(x)=\bar{g}(x)$,
for $x\neq0$ and $g(0)>0$. Since $0\in\operatorname*{dom}\bar{g}\setminus D(\partial\bar{g})$
and $\bar{g}\le g$ it follows that $\operatorname*{Graph}(\partial\bar{g})\subset\operatorname*{Graph}(\partial g)$
from which $\partial g=\partial\bar{g}\in\mathfrak{M}(\mathbb{R})$
while $g$ is not lower semicontinuous at $0$, because $g(0)>0=\liminf_{x\to0}g(x)$.

(iii) Let us consider $\phi\in\Gamma(X)$ such that $\partial\phi\in\mathfrak{M}(X)$
and 
\begin{equation}
(\operatorname*{dom}\phi\setminus D(\partial\phi))\cap\operatorname*{conv}D(\partial\phi)\neq\emptyset.\label{e2}
\end{equation}
In the case $X=\mathbb{R}^{2}$, such a function is provided by $\rho(x_{1},x_{2})=\max\{\alpha(x_{1}),|x_{2}|\}$,
where $\alpha(x_{1})=1-x_{1}^{1/2}$, if $x_{1}\ge0$; $\alpha(x_{1})=+\infty$,
if $x_{1}<0$.

The domain of $\rho$ is the right closed half-plane, whereas $D(\partial\rho)$
is this half-plane with the relative interior of the line segment
joining $(0,1)$ and $(0,-1)$ removed (see \cite[p. 218]{MR1451876}).
Recall that $\rho$ is continuous on the open right half-plane and
that $\partial\rho\in\mathfrak{M}(\mathbb{R}^{2})$, since $\rho\in\Gamma(\mathbb{R}^{2})$.
In particular, $\partial\rho$ is of NI type i.e., $\varphi_{\partial\rho}\ge c$
on $\mathbb{R}^{2}\times\mathbb{R}^{2}$.

In general $\phi$ is constructed as follows. Let $F$ be a two-dimensional
subspace of $X$ and let $H:F\to\mathbb{R}^{2}$ be a linear isomorphism
of finite-dimensional spaces. Take the closed convex process $\tilde{H}:X\rightrightarrows\mathbb{R}^{2}$,
$\tilde{H}(x)=\{H(x)\}$, for $x\in F$; $\tilde{H}(x)=\emptyset$,
for $x\in X\smallsetminus F$, and $\phi:X\to\overline{\mathbb{R}}$,
$\phi(x):=\inf\{\rho(y)\mid y\in\tilde{H}(x)\}$; more precisely 
\[
\phi(x)=(\rho\circ H)(x),\ x\in F;\ \phi(x)=+\infty,\ x\in X\setminus F.
\]

Then $\phi\in\Gamma(X)$ because $\rho\in\Gamma(\mathbb{R}^{2})$
and, by to the chain rule (see e.g. \cite[Theorem\ 2.8.6\ (iii)]{MR1921556}),
\begin{equation}
(x,x^{*})\in\operatorname*{Graph}(\partial\phi)\Leftrightarrow x\in H^{-1}(D(\partial\rho))\ {\rm and}\ \exists y\in\partial\rho(H(x)),\ x^{*}|_{F}=H^{*}(y).\label{e1}
\end{equation}
Here $H^{*}:\mathbb{R}^{2}\rightrightarrows F^{*}$ is the adjoint
of $H$ while $x^{*}|_{F}(u):=x^{*}(u)$, for $u\in F$. The last
equality in (\ref{e1}) is derived from $x^{*}\in\tilde{H}^{*}(y)$
where $\tilde{H}^{*}:\mathbb{R}^{2}\rightrightarrows X^{*}$ is the
adjoint of the process $\tilde{H}$.

Note that, according to (\ref{e1}), for every $x\in X$, $\partial\phi(x)=\partial\phi(x)+F^{\perp}$
which shows that $\operatorname*{Pr}_{X}\operatorname*{dom}\varphi_{\partial\phi}\subset F$.
Hence, for every $x\in F$, $x^{*}\in X^{*}$ 
\[
\varphi_{\partial\phi}(x,x^{*})=\varphi_{\partial\rho}(H(x),(H^{*})^{-1}(x^{*}|_{F}))\ge H(x)\cdot(H^{*})^{-1}(x^{*}|_{F})=x^{*}(x),
\]
where ``$\cdot$'' stands for the dot product in $\mathbb{R}^{2}$.
Therefore, $\varphi_{\partial\phi}\ge c$, that is, $\partial\phi$
is of NI type. Moreover, $\partial\phi$ is representable, since $\phi\in\Gamma(X)$.
Hence $\partial\phi\in\mathfrak{M}(X)$.

Condition (\ref{e2}) holds because it holds for $\rho$, $H$ is
a linear isomorphism, and $D(\partial\phi)=H^{-1}(D(\partial\rho))$,
$\operatorname*{dom}\phi=H^{-1}(\operatorname*{dom}\rho)$.

Take $\bar{x}\in(\operatorname*{dom}\phi\setminus D(\partial\phi))\cap\operatorname*{conv}D(\partial\phi)$
and let $g:X\to\mathbb{R}\cup\{+\infty\}$ be given by $g(x)=\phi(x)$,
for $x\neq\bar{x}$, $g(\overline{x})=+\infty$.

Then $g$ is not convex because $\operatorname*{dom}g=\operatorname*{dom}\phi\setminus\{\bar{x}\}$
is not convex. Indeed, $D(\partial\phi)\subset\operatorname*{dom}g$,
$\bar{x}\in\operatorname*{conv}D(\partial\phi)\subset\operatorname*{conv}(\operatorname*{dom}g)$,
while $\overline{x}\not\in\operatorname*{dom}g$.

Also, $g$ is not lower semicontinuous at $\bar{x}$. Indeed, since
$\rho\in\Gamma(\mathbb{R})$ and $\mathbb{R}$ is a Banach space,
for every $u\in\operatorname*{dom}\rho$, there exists $(u_{n})_{n}\subset D(\partial\rho)$
such that $u_{n}\to u$ in $\mathbb{R}^{2}$ and $\rho(u_{n})\to\rho(u)$
(see, for example, \cite[Theorem\ 3.1.4\ (i),\ p.\ 162]{MR1921556}).
Applying this property to $\overline{u}=H(\overline{x})$, we may
choose $(x_{n})_{n}\subset D(\partial\phi)$ such that $x_{n}\to\bar{x}$
in $X$ and $\phi(x_{n})=g(x_{n})\to\phi(\bar{x})$. Hence 
\[
\liminf_{x\to\bar{x}}g(x)\le\liminf_{n}g(x_{n})=\liminf_{n}\phi(x_{n})=\phi(\bar{x})<+\infty=g(\bar{x}).
\]
Thus, $g$ is not lower semicontinuous at $\bar{x}$.

Since $\bar{x}\not\in D(\partial\phi)$ and $\phi\le g$, we have
$\mathfrak{M}(X)\ni\partial\phi\subset\partial g\in\mathcal{M}(X)$.
It follows that $\partial g=\partial\phi\in\mathfrak{M}(X)$.

(iv) Let $f:X\rightarrow\overline{\mathbb{R}}$, and set $\phi:=\operatorname*{cl}\operatorname*{conv}f$.
Recall that, in general $\phi\leq f$. Moreover, whenever $f\in\Gamma(X)$
or $\partial f\in\mathfrak{M}(X)$, we have $\partial f=\partial\phi$
and, for every $x\in D(\partial f)=D(\partial\phi)$, $f(x)=\phi(x)$
(see subpoint (i)).

Assume first that $X$ is a Banach space. If $f$ is lower semicontinuous
with $\partial f\in\mathfrak{M}(X)$ then for every $x\in\operatorname*{dom}\phi$
$(\supset\operatorname*{dom}f)$, there exists a sequence $(x_{n})\subset D(\partial\phi)$
such that $x_{n}\rightarrow x$ and $\phi(x_{n})\rightarrow\phi(x).$
Since $f(x_{n})=\phi(x_{n})$ and $f$ is lower semicontinuous we
get $f(x)\leq\liminf_{n}f(x_{n})=\lim_{n}\phi(x_{n})=\phi(x)$; whence
$f\le\phi.$ It follows that $f=\phi\in\Gamma(X).$

In this context, the converse is a well-known fact, first proved in
\cite[Theorem\ A]{MR0262827}.

Assume second that $f$ is continuous at some $\overline{x}\in\operatorname*{int}(\operatorname*{dom}f)$,
$f$ is lower semicontinuous, and $\partial f\in\mathfrak{M}(X).$
Since $\phi\leq f$, $\phi\in\Gamma(X)$, and $f$ is continuous at
$\overline{x},$ we know that $\phi$ is continuous on $\operatorname*{int}(\operatorname*{dom}\phi)\subset D(\partial\phi)=D(\partial f)\subset\operatorname*{dom}f\subset\operatorname*{dom}\phi$.
Since $f=\phi$ on $D(\partial\phi)=D(\partial f)$, we note that
$f=\phi$ on $\operatorname*{int}(\operatorname*{dom}\phi)=\operatorname*{int}(\operatorname*{dom}f).$

For every $x\in\operatorname*{dom}\phi$ $(\supset\operatorname*{dom}f)$,
$x_{\lambda}:=(1-\lambda)x+\lambda\overline{x}\in\operatorname*{int}(\operatorname*{dom}f)$
for all $\lambda\in(0,1],$ and so $f(x_{\lambda})=\phi(x_{\lambda}).$
Since $f$ is lower semi\-continuous and $\phi|_{[x,\overline{x}]}$
is continuous (see e.g. \cite[Proposition 2.1.6, p. 50]{MR1921556}),
we get $f(x)\leq\liminf_{\lambda\rightarrow0^{+}}f(x_{\lambda})=\liminf_{\lambda\rightarrow0^{+}}\phi(x_{\lambda})=\phi(x)$,
and so $f\le\phi.$ Therefore $f=\phi\in\Gamma(X).$

Under the current assumptions, the converse is precisely \cite[Theorem\ 9, p. 282]{MR3492118};
see also Theorem \ref{f-cont} below.

(v) Since $\partial(\operatorname*{cl}\operatorname*{conv}f)\in\mathfrak{M}(X)$,
it follows from (i) that $\operatorname*{cl}\operatorname*{conv}f$
is proper. Hence $f$ cannot take the value $-\infty$. If $\operatorname*{dom}f=\emptyset$,
then $\operatorname*{Epi}f=\emptyset$, and therefore $\operatorname*{cl}\operatorname*{conv}f=+\infty$,
a contradiction. Thus $f$ is proper.

Consider a closed convex $C\subset X$ with $0\in\operatorname*{int}C$.
Then, according to \cite[Theorem\ 9,\ p.\ 282]{MR3492118}, $N_{C}=\partial\iota_{C}\in\mathfrak{M}(X)$.
Consider $g:X\to\overline{\mathbb{R}}$, $g(x)=0$, for $x\in C\setminus\{0\}$,
$g(0)>0$, $g(x)=+\infty$, for $x\in X\setminus C$. Then 
\[
\operatorname*{Epi}g=((C\setminus\{0\})\times[0,+\infty))\cup(\{0\}\times[g(0),+\infty))
\]
is neither closed nor convex because, for example, $(0,0)\in(\operatorname*{cl}\operatorname*{Epi}g\cap\operatorname*{conv}\operatorname*{Epi}g)\setminus\operatorname*{Epi}g$;
also $(\operatorname*{cl})\operatorname*{conv}\operatorname*{Epi}g=C\times[0,+\infty)$.
Therefore $g$ is neither convex nor lower semicontinuous and $\operatorname*{cl}\operatorname*{conv}g=\iota_{C}$
so $\partial(\operatorname*{cl}\operatorname*{conv}g)\in\mathfrak{M}(X)$.
Since $g(0)>0$, $\partial g(0)=\emptyset$ and so $\operatorname*{Graph}(\partial g)\subsetneq\operatorname*{Graph}(\partial(\operatorname*{cl}\operatorname*{conv}g))$.
\end{proof}

\begin{corollary} Let $X$ be a locally convex space and let $f:X\to\overline{\mathbb{R}}$.
Assume that $X$ is a Banach space or that $f$ is continuous at some
$\overline{x}\in\operatorname*{dom}f$. If $\partial f\in\mathfrak{M}(X)$
then $\operatorname*{cl}f$ is convex and $\partial(\operatorname*{cl}f)=\partial f\in\mathfrak{M}(X)$.
\end{corollary}

\begin{proof} From $\mathfrak{M}(X)\ni\partial f\subset\partial(\operatorname*{cl}f)\in\mathcal{M}(X)$
we get that $\partial(\operatorname*{cl}f)=\partial f\in\mathfrak{M}(X)$
and we use subpoint (iv) of Theorem \ref{dfdom} for $\operatorname*{cl}f$
to conclude. \end{proof}

\strut

If one increases the values of a function whose convex subdifferential
is maximal monotone outside its subdifferentiability domain, then
its maximal monotone subdifferential remains unchanged. As noted above,
however, this operation may destroy convexity and/or lower semicontinuity,
even when the original function possesses these properties. More precisely,
\begin{equation}
\partial f\in\mathfrak{M}(X),\ g=f\ {\rm on}\ D(\partial f),\ g\ge f\Rightarrow\partial g=\partial f\in\mathfrak{M}(X).\label{e3}
\end{equation}

Therefore, if one seeks conditions guaranteeing that the subdifferential
of a function is maximal monotone, the values of the function on its
subdifferentiability domain must be essential in determining the function.
This motivates the introduction of upper envelopes (see \ref{eq:}
below).

In Banach spaces, the subdifferentiability domain of every proper
lower semicontinuous convex function is dense in its domain (see for
example \cite[Theorem\ 3.1.4 (i), p. 162]{MR1921556}). As a consequence,
parts (i) and (iv) of Theorem \ref{dfdom} can be sharpened.

\begin{proposition} \label{Ba} Let $X$ be a Banach space and let
$f:X\to\overline{\mathbb{R}}$ be such that $\partial f\in\mathfrak{M}(X)$.
Then $D(\partial f)$ is dense in $\operatorname*{dom}f$. Moreover,
if $f$ is lower semicontinuous at some $x\in\operatorname*{dom}(\operatorname*{cl}\operatorname*{conv}f)$
then there exists $(x_{n})_{n}\subset D(\partial f)$ such that $x_{n}\to x$
in $X$ and $f(x_{n})\to f(x)=(\operatorname*{cl}\operatorname*{conv}f)(x)$.
In particular 
\[
\begin{aligned}D(\partial f) & \subset\{x\in\operatorname*{dom}(\operatorname*{cl}\operatorname*{conv}f)\mid f\ {\rm is\ lower\ semicontinuous\ at}\ x\}\\
 & \subset\{x\in\operatorname*{dom}f\mid f(x)=(\operatorname*{cl}\operatorname*{conv}f)(x)\}.
\end{aligned}
\]
\end{proposition}

\begin{proof} Let $\phi:=\operatorname*{cl}\operatorname*{conv}f$.
We know from Theorem \ref{dfdom} (i) that $f$, $\phi$ are proper.
Also $\phi\leq f$, $\emptyset\neq\operatorname*{dom}f\subset\operatorname*{dom}\phi$,
$\partial f=\partial\phi$, and $f=\phi$ on $D(\partial f)$. Since
$D(\partial\phi)=D(\partial f)$ is dense in $\operatorname*{dom}\phi$,
it is also dense in $\operatorname*{dom}f$. Moreover, since $\phi\in\Gamma(X)$,
for every $x\in\operatorname*{dom}\phi$, there exists $(x_{n})_{n}\subset D(\partial f)$
such that $x_{n}\to x$ in $X$ and $f(x_{n})=\phi(x_{n})\to\phi(x)$.

For every $x\in D(\partial f)\subset\operatorname*{dom}f\subset\operatorname*{dom}\phi$
and every $x_{n}\to x$ in $X$ we have 
\[
f(x)=\phi(x)\le\liminf_{n\to\infty}\phi(x_{n})\le\liminf_{n\to\infty}f(x_{n}),
\]
i.e., $f$ is lower semicontinuous at $x$.

As previously seen, whenever $f$ is lower semicontinuous at $x$,
$f(x)\leq\liminf_{n}f(x_{n})=\lim_{n}\phi(x_{n})=\phi(x)$; whence
$f(x)=\phi(x).$ \end{proof}

\strut

The results of this section indicate that the natural class of functions
in which one may seek conditions guaranteeing maximal monotonicity
of the convex subdifferential is the class of proper lower semicontinuous
convex functions. Outside this class, as the examples in this section
show, no such characterization can be expected.

\section{Upper envelopes}

Let $f:X\to\overline{\mathbb{R}}$. To $f$ we associate the function
$f^{\cup}:X\to\overline{\mathbb{R}}$ defined as the \emph{upper envelope}
of the continuous affine functions generated by $\partial f$, that
is, $f^{\cup}\coloneqq(f^{*}+\iota_{R(\partial f)})^{*}$.

Equivalently, for every $x\in X$, 
\begin{equation}
f^{\cup}(x)\coloneqq\sup\{\langle x-a,a^{*}\rangle+f(a)\mid(a,a^{*})\in\operatorname*{Graph}(\partial f)\}=\sup\{\langle x,a^{*}\rangle-f^{*}(a^{*})\mid a^{*}\in R(\partial f)\};\label{eq:}
\end{equation}
(see \cite[(B), p. 606]{MR0178103}).

The \emph{portable hull of} $f$ is defined by 
\begin{equation}
f^{\#}:=f^{\cup}+\iota_{(\operatorname*{dom}f)^{\#}}.\label{eq:-1}
\end{equation}

Here for $C\subset X$, $C^{\#}$ denotes the portable hull of $C$
which is the intersection of all the supporting half-spaces that contain
$C$ and are supported at points in $C$ (see \cite{Voisei2019})
\begin{equation}
C^{\#}:=\{x\in X\mid\forall(a,a^{*})\in\operatorname*{Graph}N_{C},\ \langle x-a,a^{*}\rangle\le0\}.\label{eq:-5}
\end{equation}
Recall that $C\subset X$ is \emph{portable} if $C=C^{\#}$ and in
general if $\operatorname*{Supp}C=\{x\in C\mid N_{C}(x)\neq\{0\}\}$
then $\operatorname*{Supp}\emptyset=\operatorname*{Supp}X=\emptyset$,
$\emptyset^{\#}=X$, 
\[
x\in C^{\#}\Leftrightarrow\forall a\in\operatorname*{Supp}C,\forall a^{*}\in N_{C}(a),\ \langle x-a,a^{*}\rangle\le0,
\]
and, for every $C\subset X$, one has $C^{\#}\neq\emptyset$ and $C\subset\operatorname*{cl}\operatorname*{conv}C\subset(\operatorname*{cl}_{\tau}\operatorname*{conv}C)^{\#}\subset C^{\#}$;
while $C^{\#}=X$ when $\operatorname*{Supp}C=\emptyset$, e.g., $X^{\#}=X$.

\strut 
\begin{center}
\begin{tikzpicture}   
\node at (-.5,0) {$C$}; \node at (1.7,1.2) {$N_C(x)$}; 
\node at (-1.4,1.4) {$C^{\#}$};

\draw[-] (-2,-1.81) -- (1,-1.81);  \draw[->] (-.84,-1.81) -- (-.84,-2.1);   
\draw[-] (.27,-2) -- (1,1);  \draw[<-] (1.6,0.854) -- (1,1);  
\draw[-] (-2,2) -- (1,1); \draw[<-] (1.3,1.9) -- (1,1);  
\draw[-] (-1.76,-2) -- (-1.76,2); \draw[->] (-1.76,-.4) -- (-2.3,-.4);
\draw[thick] 
plot [smooth,tension=1.5] coordinates{(1,1) (0.4,-0.2) (0.6,-0.2) (0.2,-1.4) (-0.5,-1.1) (-1,-1.6) (-1.4,0.4) (1,1)}; 
\end{tikzpicture} 
\par\end{center}

\begin{lemma} For every $\emptyset\neq C\subset X$, $\iota_{C}^{\cup}:=(\iota_{C})^{\cup}=\iota_{C^{\#}}=(\iota_{C})^{\#}=:\iota_{C}^{\#}$.
\end{lemma}

\begin{proof} If $x\in\operatorname*{dom}\iota_{C}^{\cup}$ that
is $\iota_{C}^{\cup}(x)=\sup\{\langle x-a,a^{*}\rangle\mid a^{*}\in N_{C}(a),\ a\in C\}<+\infty$
then, since $N_{C}(a)$ is a cone, for every $(a,a^{*})\in\operatorname*{Graph}N_{C}$,
$\langle x-a,a^{*}\rangle\le0$, i.e., $x\in C^{\#}$; in which case
$\iota_{C}^{\cup}(x)=0$. Therefore $\operatorname*{dom}\iota_{C}^{\cup}=C^{\#}$
and $\iota_{C}^{\cup}=\iota_{C^{\#}}=\iota_{C}^{\#}$. \end{proof}

\strut

Throughout this article, we use the same notation $E^{\#}$ for the
portable hull of a set $E\subset X\times\mathbb{R}$. The notation
$f^{\#}$ is justified by its relation to the epigraph of $f$, as
will be seen in the next result.

\begin{theorem} \label{fcd}Let $X$ be a locally convex space and
let $f:X\to\overline{\mathbb{R}}$. Then

\medskip{}

\emph{(i)} $f^{\cup}\le f^{\#}\le f$ and, for every $x\in D(\partial f)$,
$f^{\cup}(x)=f^{\#}(x)=f(x)$;

\medskip{}

\emph{(ii) }$f^{\cup}$ is proper if and only if $f^{\#}$ is proper
if and only if $\operatorname*{Graph}(\partial f)\neq\emptyset$;
in which case $f^{\cup},f^{\#}\in\Gamma(X)$;

\medskip{}

\emph{(iii)} \emph{$\operatorname*{Epi}(f^{\cup})\neq\emptyset$};
moreover 
\begin{equation}
(\operatorname*{Epi}f)^{\#}=\operatorname*{Epi}(f^{\cup})\cap((\operatorname*{dom}f)^{\#}\times\mathbb{R})=\operatorname*{Epi}(f^{\#})\neq\emptyset;\label{eq:-4}
\end{equation}

\emph{(iv)} $\partial f^{\cup}|_{D(\partial f)}=\partial f^{\#}|_{D(\partial f)}=\partial f$,
in particular 
\begin{equation}
\operatorname*{Graph}(\partial f)\subset\operatorname*{Graph}(\partial f^{\cup}|_{(\operatorname*{dom}f)^{\#}})\subset\operatorname*{Graph}(\partial f^{\#})\cap\operatorname*{Graph}(\partial f^{\cup});\label{eq:-7}
\end{equation}

\emph{(v)} $f^{\cup\cup}:=(f^{\cup})^{\cup}=f^{\cup}=(f^{\cup})^{\#}=:f^{\cup\#}$,
$f^{\#\#}:=(f^{\#})^{\#}=f^{\#}$;

\medskip{}

\emph{(vi)} $f=f^{\#}$ if and only if $N_{\operatorname*{Epi}f}\in\mathfrak{M}(X\times\mathbb{R})$,

\medskip{}

\emph{(vii)} $N_{\operatorname*{Epi}(f^{\cup})},N_{\operatorname*{Epi}(f^{\#})}\in\mathfrak{M}(X\times\mathbb{R})$
and $(\operatorname*{Epi}(f^{\cup}))^{\#}=\operatorname*{Epi}(f^{\cup})$,

\medskip{}

\emph{(viii)} $\partial f=\partial f^{\cup}$ if and only if $D(\partial f)=D(\partial f^{\cup})$;
$\partial f=\partial f^{\#}$ if and only if $D(\partial f)=D(\partial f^{\#})$,

\medskip{}

\emph{(ix)} For $n\ge2$, let $f^{n\cup}:X\to\overline{\mathbb{R}}$
be defined by 
\[
\begin{aligned}f^{n\cup}(x):=\sup\{\langle x-a_{1},a_{1}^{*}\rangle+\langle a_{1}-a_{2},a_{2}^{*}\rangle+\ldots & +\langle a_{n-1}-a_{n},a_{n}^{*}\rangle+f(a_{n})\\
 & \mid(a_{k},a_{k}^{*})\in\operatorname*{Graph}(\partial f),\ k=\overline{1,n}\}.
\end{aligned}
\]
Then, for every $n\ge2$, $f^{n\cup}=f^{\cup}$. \end{theorem}

\begin{proof} (i) $f^{\cup}\le f$ comes from the subdifferential
definition or the convention $\sup\emptyset=-\infty$ when $\operatorname*{Graph}(\partial f)=\emptyset$;
while, for every $x\in D(\partial f)$, we may pick $a=x$ to get
$f^{\cup}(x)\ge f(x)$. Together with $\operatorname*{dom}f\subset(\operatorname*{dom}f)^{\#}$
we see that $f^{\cup}\le f^{\#}\le f+\iota_{\operatorname*{dom}f}=f$
and so, for every $x\in D(\partial f)$, $f^{\cup}(x)=f^{\#}(x)=f(x)$.

(ii) If $\operatorname*{Graph}(\partial f)=\emptyset$ then $f^{\cup}$
is improper since it is identically equal to $-\infty$ while $f^{\#}$
is improper because it is $-\infty$ on the non-empty set $(\operatorname*{dom}f)^{\#}$.

Conversely, if $\operatorname*{Graph}(\partial f)\not=\emptyset$
then, for every $a\in D(\partial f)$, $f^{\cup}(a)=f^{\#}(a)=f(a)\in\mathbb{R}$.
Fix $(a,a^{*})\in\operatorname*{Graph}(\partial f)$. For every $x\in X$,
$f^{\#}(x)\ge f^{\cup}(x)\ge\langle x-a,a^{*}\rangle+f(a)>-\infty$.

(iii) Assume that $\operatorname*{Epi}(f^{\cup})=\emptyset$, i.e.,
$\operatorname*{dom}(f^{\cup})=\emptyset$; whence $f^{\cup}$, $f^{\#}$,
and $f$ are identically equal to $+\infty$. Therefore $\operatorname*{dom}f=\emptyset$,
$\operatorname*{Graph}(\partial f)=\emptyset$ and so $f^{\cup}$
is identically equal to $-\infty$; an obvious contradiction.

If $\operatorname*{Epi}f=\emptyset$, that is, $f$ is identically
equal to $+\infty$ then $\operatorname*{dom}f=D(\partial f)=\emptyset$,
$f^{\cup}$, $f^{\#}$ are identically equal to $-\infty$, and in
this case 
\[
\operatorname*{Epi}(f^{\#})=\operatorname*{Epi}(f^{\cup})\cap((\operatorname*{dom}f)^{\#}\times\mathbb{R})=(\operatorname*{Epi}f)^{\#}=X\times\mathbb{R}.
\]

When $\operatorname*{Epi}f\not=\emptyset$, a description of $N_{\operatorname*{Epi}f}$
is known; for $(a,t)\in\operatorname*{Epi}f$ 
\begin{equation}
N_{\operatorname*{Epi}f}(a,t)=\left\{ \begin{array}{l}
\{(a^{*},0)\mid a^{*}\in N_{\operatorname*{dom}f}(a)\},\ \mathrm{if}\ t>f(a),\\
\{(a^{*},0)\mid a^{*}\in N_{\operatorname*{dom}f}(a)\}\cup\{\lambda(u^{*},-1)\mid\lambda\ge0,\ u^{*}\in\partial f(a)\},\ \mathrm{if}\ t=f(a);
\end{array}\right.\label{nepi}
\end{equation}
in particular $(a^{*},0)\in N_{\operatorname*{Epi}f}(a,t)$ if and
only if $a^{*}\in N_{\operatorname*{dom}f}(a)$.

Therefore $(x,v)\in(\operatorname*{Epi}f)^{\#}$ if and only if for
every $(a,t)\in\operatorname*{Epi}f$, $(a^{*},\alpha)\in N_{\operatorname*{Epi}f}(a,t)$,
\[
\langle(x,v)-(a,t),(a^{*},\alpha)\rangle\le0,\ {\rm or}\ \langle x-a,a^{*}\rangle+\alpha(v-t)\le0.
\]
For $\alpha=0$ we get that $\langle x-a,a^{*}\rangle\le0$, for every
$a^{*}\in N_{\operatorname*{dom}f}(a)$, i.e., $x\in(\operatorname*{dom}f)^{\#}$.

\noindent For $\alpha<0$ we get, according to (\ref{nepi}), that
$\langle x-a,u^{*}\rangle+f(a)\le v$, for every $u^{*}\in\partial f(a)$,
i.e., $f^{\cup}(x)\le v$, that is, $(x,v)\in\operatorname*{Epi}(f^{\cup})$.

Hence $(\operatorname*{Epi}f)^{\#}=\operatorname*{Epi}(f^{\cup})\cap((\operatorname*{dom}f)^{\#}\times\mathbb{R})=\operatorname*{Epi}(f^{\#})$.

(iv) If $D(\partial f)=\emptyset$ then $\partial f^{\cup}|_{D(\partial f)}=\partial f^{\#}|_{D(\partial f)}=\partial f$
follows trivially. Otherwise, for every $(x,x^{*})\in\partial f^{\cup}$
with $x\in D(\partial f)$, and for every $y\in X$, we have $f^{\cup}(x)=f^{\#}(x)=f(x)$
and $\langle y-x,x^{*}\rangle\le f^{\cup}(x)-f(x)\le f^{\#}(y)-f(x)\le f(y)-f(x)$,
i.e., $x^{*}\in\partial f(x)$. Hence $\operatorname*{Graph}(\partial f^{\cup}|_{D(\partial f)})\subset\operatorname*{Graph}(\partial f^{\#}|_{D(\partial f)})\subset\operatorname*{Graph}(\partial f)$.

Conversely, for every $(x,x^{*})\in\operatorname*{Graph}(\partial f)$
and $y\in X$ we have $f^{\cup}(x)=f(x)$ and, from the definition
of $f^{\cup}$, $\langle y-x,x^{*}\rangle+f(x)\le f^{\cup}(y)$. Hence
$x^{*}\in\partial f^{\cup}(x)$.

(v) For every $x\in X$, 
\[
\begin{aligned}f^{\cup\cup}(x):= & \sup\{\langle x-a,a^{*}\rangle+f^{\cup}(a)\mid(a,a^{*})\in\operatorname*{Graph}(\partial f^{\cup})\}\\
\ge & \sup\{\langle x-a,a^{*}\rangle+f(a)\mid(a,a^{*})\in\operatorname*{Graph}(\partial f)\}=f^{\cup}(x),
\end{aligned}
\]
while the converse inequality $f^{\cup\cup}\le f^{\cup}$ has been
noted in (i). Hence $f^{\cup\cup}=f^{\cup}$. Combined with $f^{\cup\cup}\le f^{\cup\#}\le f^{\cup}$
we get $f^{\cup}=f^{\cup\#}$.

According to \cite[Proposition 5, p. 814]{Voisei2019} $(\operatorname*{Epi}f)^{\#\#}=(\operatorname*{Epi}f)^{\#}$.
Hence

$\operatorname*{Epi}(f^{\#\#})=(\operatorname*{Epi}(f^{\#}))^{\#}=(\operatorname*{Epi}f)^{\#\#}=(\operatorname*{Epi}f)^{\#}=\operatorname*{Epi}(f^{\#})$
and so $f^{\#\#}=f^{\#}$.

(vi) is a direct consequence of (iii) and \cite[Theorem 2, p. 814]{Voisei2019}.

(vii) follows directly from (v) and (vi).

(viii) is a direct consequence of (iv).

(ix) If we pick $(a_{k},a_{k}^{*})=(a,a^{*})\in\operatorname*{Graph}(\partial f)$,
$k=\overline{1,n}$, then we see that $f^{n\cup}\ge f^{\cup}$. Conversely,
for every $(a_{k},a_{k}^{*})\in\operatorname*{Graph}(\partial f)$,
$k=\overline{1,n}$, 
\[
\sum_{k=1}^{n-1}\langle a_{k}-a_{k+1},a_{k+1}^{*}\rangle\le\sum_{k=1}^{n-1}f(a_{k})-f(a_{k+1})=f(a_{1})-f(a_{n}),
\]
from which $f^{n\cup}\le f^{\cup}$. \end{proof}

\strut

Let us call a function $f$ \emph{portable} if $f=f^{\#}$. As previously
seen 
\begin{itemize}
\item For every $f$, the envelopes $f^{\cup}$, $f^{\#}$ are portable. 
\item $f$ is portable if and only if $\operatorname*{Epi}f$ is portable. 
\end{itemize}
\begin{remark} \emph{Recall that the} partial portable hull\emph{
of $C\subset X$ on $G\subset\operatorname*{Graph}(N_{C})$ is } 
\begin{equation}
C_{G}^{\#}:=\{x\in X\mid\forall(a,a^{*})\in G,\ \langle x-a,a^{*}\rangle\le0\}.\label{def-ph}
\end{equation}

\emph{In this case $\operatorname*{Epi}(f^{\cup})$ is the intersection
of all the supporting semi-spaces that contain $\operatorname*{Epi}f$
determined by non-vertical support hyperplanes to $\operatorname*{Epi}f$,
that is, $\operatorname*{Epi}(f^{\cup})$ has the form of a partial
portable hull of $\operatorname*{Epi}f$ and is given by $(x,v)\in\operatorname*{Epi}(f^{\cup})$
}if and only if\emph{ 
\begin{equation}
\forall a\in\operatorname*{dom}f,\ \forall(a^{*},\alpha)\in N_{\operatorname*{Epi}f}(a,f(a))\ {\rm with}\ \alpha\neq0,\ \langle(x,v)-(a,f(a)),(a^{*},\alpha)\rangle\le0,\label{epi-fu}
\end{equation}
or $\operatorname*{Epi}(f^{\cup})=(\operatorname*{Epi}f)_{G}^{\#}$,
where 
\[
G:=\{(a,f(a);a^{*},-1)\mid(a,a^{*})\in\operatorname*{Graph}(\partial f)\}\subset\operatorname*{Graph}(N_{\operatorname*{Epi}f}).
\]
} \end{remark}

With respect to the dual system $(X^{*},X)$ the dual notions are
similarly defined as follows. The \emph{upper envelope} for $g:(X^{*},w^{*})\to\overline{\mathbb{R}}$
is 
\begin{equation}
g^{\cup}(x^{*}):=\sup\{\langle x^{*}-a^{*},a\rangle+g(a^{*})\mid(a^{*},a)\in\operatorname*{Graph}(\partial g)\}=\sup\{\langle x^{*},a\rangle-g^{*}(a)\mid a\in R(\partial g)\},\label{diez-in-X*}
\end{equation}
for $K\subset X^{*}$ the \emph{portable hull} of $K$ is 
\[
K^{\#}:=\{x^{*}\in X^{*}\mid\forall(a,a^{*})\in\operatorname*{Graph}(\partial\sigma_{K}),\ \langle a,x^{*}-a^{*}\rangle\ge0\},
\]
and the \emph{portable hull} of $g$ is $g^{\#}:=g^{\cup}+\iota_{(\operatorname*{dom}g)^{\#}}$.

Naturally, for $f:X\to\overline{\mathbb{R}}$, one can consider 
\[
f^{*\cup}:=(f^{*})^{\cup},\ f^{\cup*}:=(f^{\cup})^{*},\ f^{\circ}:=f^{*\cup*}=(f^{*\cup})^{*}.
\]
These functions are proper if and only if $\operatorname*{Graph}(\partial f)\neq\emptyset$
in which case their expanded forms are 
\begin{equation}
\begin{aligned}f^{\circledast}(x^{*})\coloneqq f^{*\cup}(x^{*}) & =\sup\{\langle x^{*}-a^{*},a\rangle+f^{*}(a^{*})\mid(a,a^{*})\in\operatorname*{Graph}(\partial f)\}\\
 & =\sup\{\langle x^{*},a\rangle-f(a)\mid a\in D(\partial f)\}=(f+\iota_{D(\partial f)})^{*}(x^{*}),\ x^{*}\in X^{*};
\end{aligned}
\label{*diez}
\end{equation}

\begin{equation}
f^{\circ}=(f+\iota_{D(\partial f)})^{**}=\operatorname*{cl}\operatorname*{conv}(f+\iota_{D(\partial f)}).\label{eq:-6}
\end{equation}
In particular $f^{*\cup}(0)=-\inf_{D(\partial f)}f$.

\begin{remark} The parallel between the normal cone and the convex
subdifferential breaks down in several fundamental ways.

First, the Fitzpatrick function of the normal cone $N_{C}$ is (see
\cite[Theorem 1, p. 813]{Voisei2019}) 
\begin{equation}
\varphi_{N_{C}}(x,x^{*})=\iota_{C^{\#}}(x)+\sigma_{C}(x^{*}).\label{eq:-8}
\end{equation}

In general, unlike in certain special cases, the Fitzpatrick function
of a subdifferential need not have separated variables. Take $f(x)=\tfrac{1}{2}x^{2}$,
$x\in\mathbb{R}$ for which $\operatorname*{Graph}(\partial f)=\{(x,x)\mid x\in\mathbb{R}\}$.
We have $\partial f\in\mathfrak{M}(\mathbb{R})$ and $\varphi_{\partial f}(x,y)=\frac{1}{4}(x+y)^{2}$,
$x,y\in\mathbb{R}$, while $f^{\cup}(x)=f^{*}(x)=f(x)=\tfrac{1}{2}x^{2}$,
$x\in\mathbb{R}$.

Second, for every $f:X\to\overline{\mathbb{R}}$, $x\in X$, $x^{*}\in X^{*}$
\begin{equation}
\varphi_{\partial f}(x,x^{*})\le f^{\cup}(x)+f^{\circledast}(x^{*}).\label{impr-ineq}
\end{equation}
Indeed, for every $(a,a^{*})\in\operatorname*{Graph}(\partial f)$
\[
\langle x-a,a^{*}\rangle+\langle a,x^{*}\rangle=\langle x-a,a^{*}\rangle+f(a)+\langle a,x^{*}-a^{*}\rangle+f^{*}(a^{*});
\]
and after we pass to supremum over $(a,a^{*})\in\operatorname*{Graph}(\partial f)$
we get (\ref{impr-ineq}).

The inequality in (\ref{impr-ineq}) can be strict for certain $x\in X$,
$x^{*}\in X^{*}$, e.g., for $f(x)=\tfrac{1}{2}x^{2}$, and $x,y\in\mathbb{R}$,
$x\neq y$. However for $f=\iota_{C}$ or $f=\sigma_{C}$ for $C$
closed convex, (\ref{impr-ineq}) is an identity.

Third, even though $f^{\#\#}=f^{\#}$, $f^{\cup\cup}=f^{\cup}$ one
cannot expect that $\partial f^{\cup}$ or $\partial f^{\#}$ to be
maximal monotone . Note that in every locally convex space $X$ and
for every closed convex $C\subset X$, $\sigma_{C}:(X^{*},w^{*})\to\overline{\mathbb{R}}$
has $\sigma_{C}=\sigma_{C}^{\cup}=\sigma_{C}^{\#}$ while $\partial\sigma_{C}^{(\cup)(\#)}=N_{C}{}^{-1}\in\mathfrak{M}(X^{*})$
if and only if $C=C^{\#}$ (see \cite[Theorem 2, p. 814]{Voisei2019}).

Indeed, for every $x^{*}\in X^{*}$ 
\[
\sigma_{C}^{\cup}(x^{*})=\sup\{\langle x^{*}-a^{*},a\rangle+\sigma_{C}(a^{*})\mid(a,a^{*})\in\operatorname*{Graph}N_{C}\}\ge\sigma_{C}(x^{*}),
\]
obtained by taking $a\in C$, $a^{*}=0$ or by using $\sigma_{C}^{\cup}=(\sigma_{C}^{*}+\iota_{R(\partial\sigma_{C})})^{*}=\sigma_{C}$.
These discrepancies come from the fact that, usually, $\sigma_{C}\neq\sigma_{C^{\#}}$.
\end{remark}

\begin{theorem} \label{fcirc} Let $X$ be a locally convex space
and let $f:X\to\overline{\mathbb{R}}$ be such that $\operatorname*{Graph}(\partial f)\neq\emptyset$.
Then

\medskip{}

\emph{(i)} $f^{\cup}\le\operatorname*{cl}\operatorname*{conv}f\le f^{\circ}\in\Gamma(X)$
and, for every $x\in D(\partial f)$, $f^{\circ}(x)=f(x)$;

\medskip{}

\emph{(ii)} $f^{\circ\circ}=f^{\circ}$, $f^{*\circ*}=f^{\cup}$,
$f^{*\cup}=f^{\circ*}$, $f^{*\circ}=f^{\cup*}$;

\medskip{}

\emph{(iii)} For every $x\in X$, $\partial f(x)=\partial f^{\circ}(x)\cap R(\partial f)$;
in particular $\operatorname*{Graph}(\partial f)\subset\operatorname*{Graph}(\partial f^{\circ})$;

\medskip{}

\emph{(iv) $x\in D(\partial f)$ }if and only if\emph{ $\partial f^{\circ}(x)\cap R(\partial f)\neq\emptyset$; }

\medskip{}

\emph{(v)} $\partial f=\partial f^{\circ}$ if and only if $R(\partial f)=R(\partial f^{\circ})$.
\end{theorem}

\begin{proof} (i) From $f^{\cup}\le f\le f+\iota_{D(\partial f)}$
and $f^{\cup}\in\Gamma(X)$ we get $f^{\cup}\le\operatorname*{cl}\operatorname*{conv}f\le f^{\circ}=\operatorname*{cl}\operatorname*{conv}(f+\iota_{D(\partial f)})\le f+\iota_{D(\partial f)}$
and so $\operatorname*{cl}\operatorname*{conv}f\in\Gamma(X)$, $f^{*}\in\Gamma_{w^{*}}(X^{*})$.
The second part in (i) follows now from $f^{\cup}=f$ on $D(\partial f)$.

The proof of (ii) is plain and due to the biconjugation formula and
Theorem \ref{fcd} (v) applied for $f$ and $f^{*}$.

(iii) The direct inclusion is straightforward from $f^{\circ}=f$
on $D(\partial f)$ and $f\le f^{\circ}$. For the converse let $x^{*}\in\partial f^{\circ}(x)\cap R(\partial f)$.
This comes to $x\in\partial f^{*\cup}(x^{*})=\partial f^{*}(x^{*})$
because $x^{*}\in R(\partial f)=D(\partial f^{*})$ and we can use
Theorem \ref{fcd} (iv) for $f^{*}$.

(iv) and (v) are easy consequences of (iii). \end{proof}

\begin{remark} Given $X$ a locally convex space and $f:X\to\overline{\mathbb{R}}$
such that $\operatorname*{Graph}(\partial f)\neq\emptyset$ the function
$g=f+\iota_{D(\partial f)}$ satisfies $\operatorname*{dom}g=D(\partial f)=D(\partial g)$
because $f=g$ on $D(\partial f)$, $g\ge f$, and $\operatorname*{Graph}(\partial f)\subset\operatorname*{Graph}(\partial g)$.
Also, $f+\iota_{D(\partial f)}=g+\iota_{D(\partial g)}$, $g^{*}=f^{\circledast}$,
and $g^{\circ}=f^{\circ}=\operatorname*{cl}\operatorname*{conv}g\le g$
whilst $g^{\circ}\neq g$ is possible, e.g., when $D(\partial f)$
is not convex.

Moreover, since $\operatorname*{Graph}(\partial f)\subset\operatorname*{Graph}(\partial(\operatorname*{cl}\operatorname*{conv}f))$
and $f=\operatorname*{cl}\operatorname*{conv}f$ on $D(\partial f)$
\begin{equation}
f^{\cup}\le(\operatorname*{cl}\operatorname*{conv}f)^{\cup}\le\operatorname*{cl}\operatorname*{conv}f\le f.\label{eq:-9}
\end{equation}
There exists functions $f$ such that $f^{\cup}\neq(\operatorname*{cl}\operatorname*{conv}f)^{\cup}$.
Take $f=\iota_{C}$, for which $(\operatorname*{cl}\operatorname*{conv}C)^{\#}\subsetneq C^{\#}$.
For example when $C$ is an open convex proper subset of a Banach
space we have $C^{\#}=X$, $\operatorname*{cl}\operatorname*{conv}C=(\operatorname*{cl}\operatorname*{conv}C)^{\#}=\overline{C}$
while $f^{\#}=f^{\cup}=0$, $(\operatorname*{cl}\operatorname*{conv}f)^{\#}=(\operatorname*{cl}\operatorname*{conv}f)^{\cup}=f^{\circ}=\iota_{(\overline{C})^{\#}}=\iota_{\overline{C}}$,
$\operatorname*{Graph}(\partial f)=C\times\{0\}$, $\operatorname*{Graph}(\partial f^{\cup})=X\times\{0\}$,
$\partial f^{\circ}=N_{\overline{C}}$.

In general, for $C$ non-empty closed convex, $\iota_{C}^{\circ}=\iota_{C}$
while $\partial\iota_{C}^{(\circ)}=N_{C}$ can be non-maximal monotone
(see e.g. \cite{MR1216811,MR0092113}). \end{remark}

As a prelude to the remaining results of this paper, we reprove the
known result \cite[Theorem\ 9,\ p.\ 282]{MR3492118}. This proof reveals
a general argument for establishing maximality of the subdifferential
(see Theorem \ref{max-sd} below) and motivates the introduction of
another type of envelope (see (\ref{fsp}) below).

\begin{theorem} \label{f-cont} Let $(X,\tau)$ be a locally convex
space and let $f\in\Gamma(X)$ be continuous at some $x_{0}\in\operatorname*{int}(\operatorname*{dom}f)$.
Then $\partial f\in\mathfrak{M}(X)$. \end{theorem}

\begin{proof} Since $f\in\Gamma(X)$ and $\operatorname*{int}(\operatorname*{Epi}f)\neq\emptyset$
we know that $\operatorname*{Epi}f=(\operatorname*{Epi}f)^{\#}=\operatorname*{Epi}(f^{\#})$,
i.e., $f=f^{\#}=f^{\cup}+\iota_{(\operatorname*{dom}f)^{\#}}$. Hence
\begin{equation}
f\in\Gamma(X),\ \operatorname*{int}(\operatorname*{Epi}f)\neq\emptyset,\ x\in\operatorname*{dom}f\Longrightarrow f(x)=f^{\cup}(x).\label{ffd}
\end{equation}

Denote by $M_{\lambda}:=\{x\in X\mid f(x)\le\lambda\}$, $\lambda\in\mathbb{R}$,
the closed convex level sets of $f$.

Assume that $f$ is constant on $\operatorname*{dom}f$, that is,
$f=\iota_{C}+k$ for some $k\in\mathbb{R}$ and $C\subset X$ closed
convex with $\operatorname*{int}C\neq\emptyset$. According to \cite[Corollary 4, p. 814]{Voisei2019},
$\partial f=N_{C}\in\mathfrak{M}(X)$.

Assume that $f$ is non-constant on $\operatorname*{dom}f$.

For every fixed $x\in\operatorname*{dom}f$ such that $f(x)>\inf_{X}f$
we know that $\operatorname*{int}M_{f(x)}\neq\emptyset$ and, in particular,
$\partial(f+\iota_{M_{f(x)}})=\partial f+N_{M_{f(x)}}$.

Indeed, take $a\in\operatorname*{int}(\operatorname*{dom}f)$ such
$f(x)>f(a)>\inf_{X}f$. The continuity of $f$ at $a$ shows that
$a\in\operatorname*{int}M_{f(x)}$.

From $\operatorname*{int}(\operatorname*{Epi}(f+\iota_{M_{f(x)}})\neq\emptyset$,
$f+\iota_{M_{f(x)}}\in\Gamma(X)$, and (\ref{ffd}) used for $f+\iota_{M_{f(x)}}$
we see that $f+\iota_{M_{f(x)}}=(f+\iota_{M_{f(x)}})^{\cup}$ on $M_{f(x)}$.
Since $x\in M_{f(x)}$ 
\[
\begin{aligned}f(x) & =(f+\iota_{M_{f(x)}})(x)=(f+\iota_{M_{f(x)}})^{\cup}(x)\\
 & =\sup\{\langle x-a,a^{*}+n^{*}\rangle+f(a)\mid(a,a^{*})\in\operatorname*{Graph}(\partial f),\ n^{*}\in N_{M_{f(x)}}(a),\ f(a)\le f(x)\}\\
 & =\sup\{\langle x-a,a^{*}\rangle+f(a)\mid(a,a^{*})\in\operatorname*{Graph}(\partial f),\ f(a)\le f(x)\}\\
 & \le\varphi_{\partial f}(x,0)+f(x).
\end{aligned}
\]
because for every $a\in M_{f(x)}$, $0\in N_{M_{f(x)}}(a)$ and, for
every $n^{*}\in N_{M_{f(x)}}(a)$, $\langle x-a,n^{*}\rangle\le0$.

Therefore, for every $x\in\operatorname*{dom}f$, $\varphi_{\partial f}(x,0)\ge0$.
The same considerations applied to $f-x^{*}$ for an arbitrary $x^{*}\in X^{*}$
show in general that 
\begin{equation}
\begin{aligned}f\in\Gamma(X),\ \operatorname*{int}(\operatorname*{Epi}f)\neq\emptyset, & \ x\in\operatorname*{dom}f,\ x^{*}\in X^{*}\Longrightarrow\\
\varphi_{\partial f}(x,x^{*})= & \varphi_{\partial(f-x^{*})}(x,0)+\langle x,x^{*}\rangle\ge\langle x,x^{*}\rangle.
\end{aligned}
\label{figc}
\end{equation}

Fix $x_{0}\in\operatorname*{dom}f$, $x_{0}^{*}\in\operatorname*{dom}f^{*}$.
For every $x\in X$, $x^{*}\in X^{*}$ take $C^{*}$ a closed convex
weak-star neighborhood of $x_{0}^{*}$ that contains $x^{*}$ and
$x-x_{0}\in\operatorname*{dom}\sigma_{C^{*}}$; such as 
\[
C^{*}:=x_{0}^{*}+\{y^{*}\in X^{*}\mid|\langle x-x_{0},y^{*}\rangle|\le\max\{1,|\langle x-x_{0},x^{*}-x_{0}^{*}\rangle|\}\}.
\]

Since $\iota_{C^{*}}$ is continuous at $x_{0}^{*}\in\operatorname*{dom}f^{*}$,
according to \cite[Theorem\ 2.8.7,\ p.\ 126]{MR1921556} 
\[
(f^{*}+\iota_{C^{*}})^{*}(u)=(f\Square\sigma_{C^{*}})(u):=\min\{f(u-v)+\sigma_{C^{*}}(v)\mid v\in X\},\ u\in X
\]
$\operatorname*{dom}(f\Square\sigma_{C^{*}})=\operatorname*{dom}f+\operatorname*{dom}\sigma_{C^{*}}\ni x$,
$\operatorname*{int}(\operatorname*{Epi}(f\Square\sigma_{C^{*}}))\neq\emptyset$
because $f\Square\sigma_{C^{*}}\le f$, and $(a,a^{*})\in\partial(f\Square\sigma_{C^{*}})$
if and only if $a\in\partial f^{*}(a^{*})+N_{C^{*}}(a^{*})$, $a^{*}\in R(\partial f)\cap C^{*}$.
Using (\ref{figc}) for $f\Square\sigma_{C^{*}}$ together with $x^{*}\in C^{*}$
we get 
\[
\langle x,x^{*}\rangle\le\varphi_{\partial(f\Square\sigma_{C^{*}})}(x,x^{*})=\sup\{\langle x^{*}-a^{*},a\rangle+\langle x,a^{*}\rangle\mid(a,a^{*})\in\partial(f\Square\sigma_{C^{*}})\}
\]
\[
=\sup\{\langle x^{*}-a^{*},b\rangle+\langle x^{*}-a^{*},n\rangle+\langle x,a^{*}\rangle\mid a^{*}\in R(\partial f)\cap C^{*},a^{*}\in\partial f(b),n\in N_{C^{*}}(a^{*})\}
\]
\[
=\sup\{\langle x^{*}-a^{*},b\rangle+\langle x,a^{*}\rangle\mid a^{*}\in R(\partial f)\cap C^{*},a^{*}\in\partial f(b)\}\le\varphi_{\partial f}(x,x^{*}).
\]
We showed that $\partial f$ is of NI-type so $\partial f\in\mathfrak{M}(X)$
since $\partial f$ is representable; a representative being given
by $F(x,x^{*})=f(x)+f^{*}(x^{*})$, $x\in X$, $x^{*}\in X^{*}$;
$F\in\Gamma_{\tau\times w^{*}}(X\times X^{*})$. \end{proof}

\strut

As the previous proof shows, the characterization of maximal monotonicity
in \cite[Theorem\ 2.3]{MR2207807} or \cite[Theorem 24]{gm} assumes
the following special form for convex subdifferentials.

\begin{theorem} \label{max-sd} Let $X$ be a locally convex space
and let $f\in\Gamma(X)$. The following are equivalent:

\medskip{}

\noindent\emph{(i)} $\partial f\in\mathfrak{M}(X)$,

\medskip{}

\noindent\emph{(ii)} For every $(x,x^{*})\in X\times X^{*}$, there
is $(a,a^{*})\in\operatorname*{Graph}(\partial f)$ such that $\langle x-a,a^{*}-x^{*}\rangle\ge0$,
$(f-x^{*})(a)\le(f-x^{*})(x)$, and $(f^{*}-x)(a^{*})\le(f^{*}-x)(x^{*})$.

\medskip{}

\noindent\emph{(iii)} For every $(x,x^{*})\in X\times X^{*}$, $\varphi_{\partial f}(x,x^{*})\ge\langle x,x^{*}\rangle$
or $\partial f$ is of NI--type. \end{theorem}

\begin{proof} (i) $\Rightarrow$ (ii) If $(x,x^{*})\in\operatorname*{Graph}(\partial f)$
then we take $(a,a^{*})=(x,x^{*})$.

If $(x,x^{*})\notin\operatorname*{Graph}(\partial f)\in\mathfrak{M}(X)$
then there is $(a,a^{*})\in\operatorname*{Graph}(\partial f)$ such
that $\langle a-x,a^{*}-x^{*}\rangle<0$, that is $\langle a,a^{*}-x^{*}\rangle<\langle x,a^{*}-x^{*}\rangle$.
Since $(a,a^{*})\in\operatorname*{Graph}(\partial f)$, $\langle x-a,a^{*}\rangle\le f(x)-f(a)$,
or, $(f-a^{*})(a))\le(f-a^{*})(x)$. Add these latter inequalities
to get $(f-x^{*})(a)\le(f-x^{*})(x)$. Similarly for $(f^{*}-x)(a^{*})\le(f^{*}-x)(x^{*})$.

(ii) $\Rightarrow$ (iii) is due to $\varphi_{\partial f}(x,x^{*})\ge\langle x-a,a^{*}\rangle+\langle a,x^{*}\rangle\ge\langle x,x^{*}\rangle$.

(iii) $\Rightarrow$ (i) is part of the converse of \cite[Theorem\ 2.3]{MR2207807}
and it requires $f\in\Gamma(X)$ in order to ensure the representability
of $\partial f$. \end{proof}

\begin{proposition} \label{max->cup} Let $(X,\tau)$ be a locally
convex space and let $f\in\Gamma_{\tau}(X)$. If $\partial f\in\mathfrak{M}(X)$
then $f=f^{\cup}$ and $f^{*}=f^{\circledast}$. \end{proposition}

\begin{proof} Since $\partial f\in\mathfrak{M}(X)$ and $\partial f$
is cyclically monotone, it follows that $\partial f$ is maximal cyclically
monotone. Since $f\in\Gamma_{\tau}(X)$, \cite[Theorem 36]{gm}, yields
$f=f^{\cup}$ and $f^{*}=f^{\circledast}$.

Alternatively, $f^{\cup}\in\Gamma_{\tau}(X)$ and $f^{\circledast}\in\Gamma_{w^{*}}(X^{*})$
due to $\operatorname*{Graph}(\partial f)\neq\emptyset$.

As seen in Theorem \ref{max-sd} (ii), for every $x\in X$, there
is $(a,a^{*})\in\operatorname*{Graph}(\partial f)$ such that $\langle x-a,a^{*}\rangle\ge0$.
That implies $f^{\cup}(x)\ge f(a)$, so $\inf_{X}f^{\cup}=\inf_{D(\partial f)}f$
or $f^{\cup*}(0)=f^{*\cup}(0)$.

But $\partial f\in\mathfrak{M}(X)$ is equivalent to, for every $x^{*}\in X^{*}$,
$\partial(f-x^{*})\in\mathfrak{M}(X)$. Hence 
\[
(f-x^{*})^{\cup*}(0)=(f-x^{*})^{*\cup}(0).
\]
Note that, for $x\in X$, $x^{*},u^{*}\in X^{*}$, $(f-x^{*})^{\cup}(x)=f^{\cup}(x)-\langle x,x^{*}\rangle$,
$(f-x^{*})^{\cup*}(u^{*})=f^{\cup*}(u^{*}+x^{*})$, $(f-x^{*})^{*}(u^{*})=f^{*}(x^{*}+u^{*})$,
$(f-x^{*})^{*\cup}(u^{*})=f^{*\cup}(x^{*}+u^{*})$. Set $u^{*}=0$
to get $f^{\cup*}=f^{*\cup}$ followed by $f^{\cup}=f^{*\cup*}=f^{\circ}$
and $f^{\cup}=f$ (see Theorem \ref{fcirc} (i), (ii)). \end{proof}

\begin{remark} \emph{Note again that whenever }$f\in\Gamma(X)$\emph{
has }$\operatorname*{Graph}(\partial f)\neq\emptyset$ 
\begin{equation}
f=f^{\cup},\ f^{*}=f^{\circledast}\Leftrightarrow f^{\cup}=f^{\circ}.\label{eq:-10}
\end{equation}
\emph{Indeed, if $f=f^{\cup}$ and $f^{*}=f^{*\cup}$ then, by the
biconjugation formula, $f=f^{\circ}$. Conversely, if $f^{\cup}=f^{\circ}$,
from $f^{\cup}\le f\le f^{\circ}$ we know that $f=f^{\cup}$ and
$f=f^{*\cup*}$. The latter implies $f^{*}=f^{*\cup}$.} \end{remark}

\section{Maximal monotonicity via the low-upper envelope}

For $f:X\to\overline{\mathbb{R}}$ define $f^{\smallsmile}:X\to\overline{\mathbb{R}}$
by 
\begin{equation}
\begin{aligned}f^{\smallsmile}(x) & :=\sup\{\langle x-a,a^{*}\rangle+f(a)\mid(a,a^{*})\in\operatorname*{Graph}(\partial f),\ f(a)\le f(x)\}\\
 & \ =\sup\{\langle x,a^{*}\rangle-f^{*}(a^{*})\mid(a,a^{*})\in\operatorname*{Graph}(\partial f),\ f(a)\le f(x)\},\ x\in X.
\end{aligned}
\label{fsp}
\end{equation}
as the \emph{low-upper envelope} of the continuous affine functions
defined by $\partial f$.

\begin{remark} The definition of $f^{\smallsmile}$ is suggested
in the proof of Theorem \ref{f-cont} by the use of the implicit upper
envelope $(f+\iota_{M_{f(x)}})^{\cup}(x)$ which, in that case, equals
$f^{\smallsmile}(x)$ due to $\partial(f+\iota_{M_{f(x)}})=\partial f+N_{M_{f(x)}}$.
However in general $f^{\smallsmile}(x)\le(f+\iota_{M_{f(x)}})^{\cup}(x)$.
\end{remark}

In the case of a maximal monotone subdifferential the properness of
$f^{\smallsmile}$ is ensured by Theorem \ref{max-sd} (ii). The properties
of $f^{\smallsmile}$ are summarized in the next result.

\begin{theorem} \label{f-sp}Let $X$ be a locally convex space and
let $f:X\to\overline{\mathbb{R}}$. Then

\medskip{}

\emph{(i)} $f^{\smallsmile}\le f^{\cup}\le f^{\#}\le f$; for every
$x\in D(\partial f)$, $f^{\smallsmile}(x)=f(x)$ and, for every $x\not\in\operatorname*{dom}f$,
$f^{\smallsmile}(x)=f^{\cup}(x)$;

\medskip{}

\emph{(ii) }$f^{\smallsmile}$ is proper if and only if $\operatorname*{Graph}(\partial f)\neq\emptyset$
and $f^{*}(0)=f^{*\cup}(0)$; in this case $f^{\smallsmile}\in\Gamma(X)$;

\medskip{}

\emph{(iii)} For every $x\in X$, $f^{\smallsmile}(x)\le\varphi_{\partial f}(x,0)+f(x)$.

\end{theorem}

\begin{proof} Subpoints (i), (iii) are plain.

(ii) Note that 
\[
f^{*}(0)=f^{*\cup}(0)\Leftrightarrow\inf_{X}f=\inf_{D(\partial f)}f.
\]

Directly, since $f^{\smallsmile}$ is proper $\operatorname*{Graph}(\partial f)\neq\emptyset$
and so $f$ is proper. For every $x\in\operatorname*{dom}f\subset\operatorname*{dom}f^{\smallsmile}$,
there is $a\in D(\partial f)$ such that $f(a)\le f(x)$; from which
$\inf_{X}f=\inf_{D(\partial f)}f$.

Conversely, since $\operatorname*{Graph}(\partial f)\neq\emptyset$,
$f$ is proper and $f^{\smallsmile}$ is not identically equal to
$+\infty$. It remains to prove that $f^{\smallsmile}$does not take
the value $-\infty$, that is, for every $x\in X$, there is $a\in D(\partial f)$
such that $f(a)\le f(x)$.

Indeed, if $x\in\operatorname*{argmin}f$, i.e., $f(x)=\min_{X}f$
then $0\in\partial f(x)$, $x\in D(\partial f)$, and we can take
$a=x$. Otherwise, $f(x)>\inf_{X}f=\inf_{D(\partial f)}f$ and we
are done. \end{proof}

\begin{corollary} \label{spx*} Let $X$ be a locally convex space
and let $f:X\to\overline{\mathbb{R}}$. Then, for every $x^{*}\in X^{*}$,
$(f-x^{*})^{\smallsmile}$ is proper if and only if $\operatorname*{Graph}(\partial f)\neq\emptyset$
and $f^{*}=f^{\circledast}$. \end{corollary}

\begin{proof} It suffices to notice that $(f-x^{*})^{*}(0)=f^{*}(x^{*})$
and $(f-x^{*})^{*\cup}(0)=-\inf_{D(\partial f)}(f-x^{*})=f^{*\cup}(x^{*})$
and use Theorem \ref{f-sp} (ii). \end{proof}

\begin{theorem} \label{max-sd-sp} Let $(X,\tau)$ be a locally convex
space. The following are equivalent:

\medskip{}

\emph{(i)} For every $f\in\Gamma_{\tau}(X)$, $\partial f\in\mathfrak{M}(X)$;

\medskip{}

\emph{(ii)} For every $f\in\Gamma_{\tau}(X)$, $x\in\operatorname*{dom}f$,
$x^{*}\in X^{*}$, $\varphi_{\partial f}(x,x^{*})\ge\langle x,x^{*}\rangle$;

\medskip{}

\emph{(iii)} For every $f\in\Gamma_{\tau}(X)$, $x\in\operatorname*{dom}f$,
$\varphi_{\partial f}(x,0)\ge0$;

\medskip{}

\emph{(iv)} For every $f\in\Gamma_{\tau}(X)$, $x\in\operatorname*{dom}f$,
$f(x)=f^{\smallsmile}(x)$;

\medskip{}

\emph{(v)} For every $f\in\Gamma_{\tau}(X)$, $f=f^{\smallsmile}$;

\medskip{}

\emph{(vi)} For every $f\in\Gamma_{\tau}(X)$, $x\in\operatorname*{dom}f$
there is a net $\{(a_{i},a_{i}^{*})\}_{i\in I}\subset\operatorname*{Graph}(\partial f)$
such that, for every $i\in I$, $f(a_{i})\le f(x)$; $a_{i}\to x$
in $(X,\tau)$; $\lim_{i\in I}f(a_{i})=f(x)$ and $\lim_{i\in I}\langle x-a_{i},a_{i}^{*}\rangle+f(a_{i})=f(x)$;

\medskip{}

\emph{(vii)} For every $f\in\Gamma_{\tau}(X)$, $x\in\operatorname*{dom}f$
there is a net $\{(a_{i},a_{i}^{*})\}_{i\in I}\subset\operatorname*{Graph}(\partial f)$
such that $\lim_{i\in I}\langle x-a_{i},a_{i}^{*}\rangle=0$;

\medskip{}

\emph{(viii)} For every $f\in\Gamma_{\tau}(X)$, $x\in\operatorname*{dom}f$,
$\epsilon>0$ there is $(a_{\epsilon},a_{\epsilon}^{*})\in\operatorname*{Graph}(\partial f)$
such that $\langle x-a_{\epsilon},a_{\epsilon}^{*}\rangle\ge-\epsilon$.

In this case for every $f\in\Gamma_{\tau}(X)$, $\operatorname*{cl}_{\tau}D(\partial f)=\operatorname*{cl}_{\tau}(\operatorname*{dom}f)$,
$\operatorname*{cl}_{w^{*}}R(\partial f)=\operatorname*{cl}_{w^{*}}(\operatorname*{dom}f^{*})$
are convex. \end{theorem}

\begin{proof} (i) $\Rightarrow$ (ii) $\Rightarrow$ (iii), (v) $\Rightarrow$
(iv), (vi) $\Rightarrow$ (vii) $\Rightarrow$ (viii) $\Rightarrow$
(iii) are plain; (iii)$\Rightarrow$ (ii) is due to $\varphi_{\partial(f-x^{*})}(x,0)=\varphi_{\partial f}(x,x^{*})-\langle x,x^{*}\rangle$;
(iv) $\Rightarrow$ (iii) follows from Theorem \ref{f-sp} (iii).

(ii) $\Rightarrow$ (i) Since $\partial f\in\mathfrak{M}(X)\Leftrightarrow\partial g\in\mathfrak{M}(X)$
where for $(x_{0},x_{0}^{*})\in X\times X^{*}$ 
\[
g(x):=f(x+x_{0})-x_{0}^{*}(x)-f(x_{0}),\ x\in X
\]
we can assume without loss of generality that $(0,0)\in\operatorname*{Graph}(\partial f)$.

For every $(x,x^{*})\in X\times X^{*}$ take $C^{*}$ a weak-star
closed convex neighborhood of $0\in X^{*}$, such that $x\in\operatorname*{dom}\sigma_{C^{*}}$
and $x^{*}\in C^{*}$ e.g. $C^{*}:=\{y^{*}\in X^{*}\mid|\langle x,y^{*}\rangle|\le\max\{1,|\langle x,x^{*}\rangle|\}\}$.
Let $F:=(f^{*}+\iota_{C^{*}})^{*}=f\Square\sigma_{C^{*}}$ since $\iota_{C^{*}}$
is weak-star continuous at $0\in\operatorname*{dom}f^{*}$.

Hence $\operatorname*{dom}F=\operatorname*{dom}f+\operatorname*{dom}\sigma_{C^{*}}\ni x$
and $(a,a^{*})\in\partial F$ if and only if $a\in\partial f^{*}(a^{*})+N_{C^{*}}(a^{*})$,
$a^{*}\in R(\partial f)\cap C^{*}$ (see e.g. \cite[Theorem\ 2.8.7,\ p.\ 126]{MR1921556}).

We know that $\varphi_{\partial F}(x,x^{*})\ge\langle x,x^{*}\rangle$
but also that 
\[
\varphi_{\partial F}(x,x^{*})=\sup\{\langle x^{*}-a^{*},a\rangle+\langle x,a^{*}\rangle\mid(a,a^{*})\in\operatorname*{Graph}(\partial F)\}
\]
\[
=\sup\{\langle x^{*}-a^{*},b\rangle+\langle x^{*}-a^{*},n\rangle+\langle x,a^{*}\rangle\mid a^{*}\in R(\partial f)\cap C^{*},a^{*}\in\partial f(b),n\in N_{C^{*}}(a^{*})\}
\]
\[
=\sup\{\langle x^{*}-a^{*},b\rangle+\langle x,a^{*}\rangle\mid a^{*}\in R(\partial f)\cap C^{*},a^{*}\in\partial f(b)\}\le\varphi_{\partial f}(x,x^{*}).
\]

(i) $\Rightarrow$ (v) According to Proposition \ref{max->cup}, we
know that, for every $f\in\Gamma_{\tau}(X)$, $f=f^{\cup}$.

Also, according to Theorem \ref{max-sd} (ii), for every $f\in\Gamma_{\tau}(X)$,
$x\in X$, there is $(a,a^{*})\in\operatorname*{Graph}(\partial f)$
such that $\langle x-a,a^{*}\rangle\ge0$ and $f(a)\le f(x)$ which
yields that $f^{\smallsmile}(x)\ge f(a)$.

For every $x\in\operatorname*{dom}f$ and every $C\in\mathscr{V}_{\tau}^{c}(x)$
(that is, $C$ is a $\tau-$closed convex neighborhood of $x$), $\partial(f+\iota_{C})=\partial f+N_{C}\in\mathfrak{M}(X)$.

Let$(a_{C},a_{C}^{*})\in\operatorname*{Graph}(\partial f)$, be such
that $a_{C}\in D(\partial f)\cap C$ and $(f+\iota_{C})^{\smallsmile}(x)\ge f(a_{C})$.

We have 
\[
\begin{aligned}(f+\iota_{C})^{\smallsmile}(x)=\sup\{\langle x-a,a^{*}\rangle+\langle x-a,n^{*}\rangle+ & f(a)\mid(a,a^{*})\in\operatorname*{Graph}(\partial f),\\
 & a\in C,\ n^{*}\in N_{C}(a),f(a)\le f(x)\}
\end{aligned}
\]
\[
=\sup\{\langle x-a,a^{*}\rangle+f(a)\mid(a,a^{*})\in\operatorname*{Graph}(\partial f),\ a\in C,\ f(a)\le f(x)\}\le f^{\smallsmile}(x).
\]

Hence 
\[
\forall x\in\operatorname*{dom}f,\ \forall C\in\mathscr{V}_{\tau}^{c}(x),\ \exists a_{C}\in D(\partial f)\cap C:\ f^{\smallsmile}(x)\ge f(a_{C}).
\]
Clearly, $a_{C}\to x$, $f(x)\le\liminf f(a_{C})\le f^{\smallsmile}(x)$,
and so $f^{\smallsmile}(x)=f(x)$.

If $x\not\in\operatorname*{dom}f$ then $f^{\smallsmile}(x)=f^{\cup}(x)=f(x)$

(v) $\Rightarrow$ (vi) For every $f\in\Gamma_{\tau}(X)$, $x\in\operatorname*{dom}f$,
$C\in\mathscr{V}_{\tau}^{c}(x)$, we have $(f+\iota_{C})^{\smallsmile}(x)=f(x)$.
But, as previously seen 
\[
(f+\iota_{C})^{\smallsmile}(x)=\sup\{\langle x-a,a^{*}\rangle+f(a)\mid(a,a^{*})\in\partial f,\ a\in C,\ f(a)\le f(x)\}.
\]
Hence for every $C\in\mathscr{V}_{\tau}^{c}(x)$, $\epsilon>0$ there
is $(a_{C,\epsilon},a_{C,\epsilon}^{*})\in\partial f$ such that $a_{C,\epsilon}\in C$,
$f(a_{C,\epsilon})\le f(x)$ and 
\begin{equation}
f(x)-\epsilon<\langle x^{*}-a_{C,\epsilon}^{*},a_{C,\epsilon}\rangle+f(a_{C,\epsilon})\le f(x).\label{eq:-11}
\end{equation}
If we denote by $I=\mathscr{V}_{\tau}(x)\times(0,+\infty)$ with the
regular order relation $i=(C,\epsilon)\preceq i'=(C',\epsilon')\Leftrightarrow C\supset C',\ \epsilon\ge\epsilon'$
then $(I,\preceq)$ is directed and the net $\{(a_{i},a_{i}^{*})\}_{i\in I}\subset\partial f$
has $a_{i}\to x$ in $(X,\tau)$, $f(x)\le\liminf_{i}f(a_{i})\le\limsup_{i}f(a_{i})\le f(x)$,
that is, $\lim_{i\in I}f(a_{i})=f(x)$, and from (\ref{eq:-11}),
$\lim_{i\in I}\langle x-a_{i},a_{i}^{*}\rangle+f(a_{i})=f(x)$. \end{proof}

\begin{remark} \emph{In Theorem \ref{max-sd-sp}, for a fixed $f\in\Gamma_{\tau}(X)$,
none of the subpoints (ii)--(viii) imply (i). }

\emph{For a closed convex $\emptyset\neq C\subset X$, $f=\iota_{C}$
satisfies all (ii)--(iv) and (vi)-(viii) but not necessarily (i).
Indeed, we always have, for every $x\in C=\operatorname*{dom}\iota_{C}$,
$x^{*}\in X^{*}$, $\varphi_{N_{C}}(x,x^{*})=\iota_{C^{\#}}(x)+\sigma_{C}(x^{*})\ge\langle x,x^{*}\rangle$,
$\iota_{C}^{\smallsmile}=\iota_{C}^{\cup}=\iota_{C}^{\#}=\iota_{C^{\#}}$,
$\iota_{C^{\#}}(x)=\iota_{C}(x)$, and for subpoints (vi)--(viii)
we can take $a_{i}=a_{\epsilon}=x$, $a_{i}^{*}=a_{\epsilon}^{*}=0$
while $N_{C}$ can be non-maximal monotone, i.e., $C\neq C^{\#}$.}

Consider $C\subset X$ closed convex such that $0\in C$. We have
$\sigma_{C}^{\smallsmile}=\sigma_{C}\in\Gamma_{w^{*}}(X^{*})$ because,
due to $0\in C$, $\sigma_{C}\ge0$ and we can pick $a^{*}=0$ in
the definition of $\sigma_{C}^{\smallsmile}$ for an arbitrary $a\in C$.
As previously seen $\partial\sigma_{C}=N_{C}^{-1}$ is not necessarily
maximal monotone. Hence, when $C^{\#}\neq C$, $\sigma_{C}$ satisfies
(v) but not (i).

\end{remark}

Similar considerations can be performed for the extended notion of
$\epsilon-$low-upper envelope of $f:X\to\overline{\mathbb{R}}$ given
by 
\begin{equation}
\begin{aligned}f_{\epsilon}^{\smallsmile}(x) & :=\sup\{\langle x-a,a^{*}\rangle+f(a)\mid(a,a^{*})\in\operatorname*{Graph}(\partial f),\ f(a)\le f(x)+\epsilon\}\\
 & \ =\sup\{\langle x,a^{*}\rangle-f^{*}(a^{*})\mid(a,a^{*})\in\operatorname*{Graph}(\partial f),\ f(a)\le f(x)+\epsilon\},\ x\in X,\ \epsilon\ge0.
\end{aligned}
\label{fsp-eps}
\end{equation}

Note that $f_{0}^{\smallsmile}=f^{\smallsmile}$; for $0\le\epsilon_{1}\le\epsilon_{2}$,
$f_{\epsilon_{1}}^{\smallsmile}\le f_{\epsilon_{2}}^{\smallsmile}$;
\begin{equation}
f^{\smallsmile}\le\inf_{\varepsilon>0}f_{\varepsilon}^{\smallsmile}=\lim_{\varepsilon\downarrow0}f_{\varepsilon}^{\smallsmile}\le f_{\epsilon}^{\smallsmile}\le f,\ \epsilon>0;\label{eq:-12}
\end{equation}
and $\sup_{\epsilon\ge0}f_{\epsilon}^{\smallsmile}=f^{\cup}$.

\begin{theorem} \label{fspeps} Let $X$ be a locally convex space.
The following are equivalent:

\medskip{}

\emph{(i)} For every $f\in\Gamma(X)$, $\partial f\in\mathfrak{M}(X)$,

\medskip{}

\emph{(ii)} For every $f\in\Gamma(X)$, $x\in\operatorname*{dom}f$,
$f(x)=\inf_{\epsilon>0}f_{\epsilon}^{\smallsmile}(x)$,

\medskip{}

\emph{(iii)'} For every $f\in\Gamma(X)$, $\epsilon>0$, $x\in\operatorname*{dom}f$,
$f(x)=f_{\epsilon}^{\smallsmile}(x)$,

\medskip{}

\emph{(iv)'} For every $f\in\Gamma(X)$, $\epsilon>0$, $f=f_{\epsilon}^{\smallsmile}$.
\end{theorem}

\begin{proof} (iv) $\Rightarrow$ (iii) $\Rightarrow$ (ii) are plain;
(i) $\Rightarrow$ (iv) is due to $f^{\smallsmile}\le f_{\epsilon}^{\smallsmile}\le f$,
$\epsilon>0$ and Theorem \ref{max-sd-sp} (i) $\Rightarrow$ (v).

(ii) $\Rightarrow$ (iii)' From (\ref{eq:-12}) we get that, for every
$\epsilon>0$, $f=f_{\epsilon}^{\smallsmile}$ on $\operatorname*{dom}f$.

(iii)' $\Rightarrow$ (i) For every $f\in\Gamma(X)$, $x\in\operatorname*{dom}f$,
\[
f(x)=\sup\{\langle x-a,a^{*}\rangle+f(a)\mid(a,a^{*})\in\operatorname*{Graph}(\partial f),\ f(a)\le f(x)+\epsilon\}<+\infty
\]
implies the existence of $\{(a_{\epsilon},a_{\epsilon}^{*})\}_{\epsilon>0}\subset\operatorname*{Graph}(\partial f)$
such that, for every $\epsilon>0$, $\langle x-a_{\epsilon},a_{\epsilon}^{*}\rangle+f(a_{\epsilon})\ge f(x)-\epsilon$
and $f(a_{\epsilon})\le f(x)+\epsilon$; whence we get $\langle x-a_{\epsilon},a_{\epsilon}^{*}\rangle\ge-2\epsilon$
and we use Theorem \ref{max-sd-sp} (viii) $\Rightarrow$ (i). \end{proof}

\strut

We conclude this section by presenting several proofs of the maximality
of the convex subdifferential in the Banach space setting. Our aim
is to identify minimal Banach-space specific assumptions, in particular
those requiring only a limited use of Br{ø}nsted--Rockafellar's
lemma \cite[Lemma,\ p.\ 608]{MR0178103} or Ekeland's variational
principle.

\begin{theorem} \label{max-sd-spB} Let $X$ be a Banach space. For
every $f\in\Gamma(X)$, $\partial f\in\mathfrak{M}(X)$. \end{theorem}

\begin{proof}[Proof 1] Denote by ``$\left\llbracket \cdot\right\rrbracket $''
the given norm of $X$ and by ``$\langle\cdot,\cdot\rangle$'' the
duality of $X\times X^{*}$.

Let $f\in\Gamma(X)$, $x\in\operatorname*{dom}f$, $\epsilon>0$.
Take $x^{*}\in\partial_{\epsilon}f(x)$ and the equivalent norm on
$X$ 
\[
\|u\|=\left\llbracket u\right\rrbracket +|\langle u,x^{*}\rangle|,\ u\in X.
\]

According to the Br{ø}nsted--Rockafellar's Lemma \cite[Lemma,\ p.\ 608]{MR0178103}
there exists $(a_{\epsilon},a_{\epsilon}^{*})\in\operatorname*{Graph}(\partial f)$
such that $\|x-a_{\epsilon}\|\le\sqrt{\epsilon}$ and $\|x^{*}-a_{\epsilon}^{*}\|_{*}\le\sqrt{\epsilon}$.
Here ``$\|\cdot\|_{*}$'' denotes the dual norm of $X^{*}$.

In particular $|\langle x-a_{\epsilon},x^{*}\rangle|\le\sqrt{\epsilon}$
and 
\[
|\langle x-a_{\epsilon},x^{*}-a_{\epsilon}^{*}\rangle|\le\|x-a_{\epsilon}\|\|x^{*}-a_{\epsilon}^{*}\|_{*}\le\epsilon.
\]
We have $|\langle x-a_{\epsilon},a_{\epsilon}^{*}\rangle|\le|\langle x-a_{\epsilon},a_{\epsilon}^{*}-x^{*}\rangle|+|\langle x-a_{\epsilon},x^{*}\rangle|\le\epsilon+\sqrt{\epsilon}$.
Using Theorem \ref{max-sd-sp} (viii) $\Rightarrow$ (i) we get $\partial f\in\mathfrak{M}(X)$.

Alternatively, since $(a_{\epsilon},a_{\epsilon}^{*})\in\operatorname*{Graph}(\partial f)$
and $x^{*}\in\partial_{\epsilon}f(x)$ we have 
\[
f(a_{\epsilon})\le f(x)+\langle a_{\epsilon}-x,a_{\epsilon}^{*}\rangle\le f(x)+\epsilon+\sqrt{\epsilon},
\]
\[
\langle x-a_{\epsilon},a_{\epsilon}^{*}\rangle=\langle x-a_{\epsilon},a_{\epsilon}^{*}-x^{*}\rangle+\langle x-a_{\epsilon},x^{*}\rangle\ge-\epsilon+f(x)-f(a_{\epsilon})-\epsilon,
\]
\[
f(x)\ge f_{\epsilon+\sqrt{\epsilon}}^{\smallsmile}(x)\ge\langle x-a_{\epsilon},a_{\epsilon}^{*}\rangle+f(a_{\epsilon})\ge f(x)-2\epsilon;
\]
in this case Theorem \ref{fspeps} (ii) $\Rightarrow$ (i) completes
the argument. \end{proof}

\strut

\begin{proof}[Proof 2] Let $f\in\Gamma(X)$, $x\in\operatorname*{dom}f$,
$x^{*}\in X^{*}$. According to \cite[Theorems 3.1.1,\ 3.1.4 (i)]{MR1921556}
there is $(x_{n},x_{n}^{*})_{n\ge1}\subset\operatorname*{Graph}(\partial f)$
such that, for every $n\ge1$, $x_{n}^{*}\in\partial_{2n^{-2}}f(x)$,
$\|x_{n}-x\|\le n^{-1}$, $|f(x_{n})-f(x)|\le n^{-2}+n^{-1}$. For
every $n\ge1$ we have 
\[
\begin{aligned}\varphi_{\partial f}(x,x^{*})-\langle x,x^{*}\rangle & \ge\langle x-x_{n},x_{n}^{*}\rangle+\langle x_{n}-x,x^{*}\rangle\\
 & \ge f(x)-f(x_{n})-2n^{-2}-n^{-1}\|x^{*}\|\ge-3n^{-2}-n^{-1}(\|x^{*}\|+1).
\end{aligned}
\]
Let $n\rightarrow\infty$ to get $\varphi_{\partial f}(x,x^{*})\ge\langle x,x^{*}\rangle$.
By Theorem \ref{max-sd-sp} (ii)$\Rightarrow$ (i), $\partial f\in\mathfrak{M}(X)$.
\end{proof}

\strut

\begin{proof}[Proof 3] Let $f\in\Gamma(X)$, $x\in\operatorname*{dom}f$,
and $\epsilon>0$. Define $g:X\rightarrow\overline{\mathbb{R}}$ by
\[
g(y)\coloneqq f(y)+\tfrac{1}{2}\|y-x\|^{2}.
\]
Clearly, $g\in\Gamma(X)$. Moreover $g$ is bounded from below. Indeed,
since $f\in\Gamma(X)$, $f^{*}$ is proper and so, there exist $w^{*}\in X^{*}$,
$\alpha\in\mathbb{R},$ such that, for every $y\in X$, $f(y)\ge\langle y,w^{*}\rangle+\alpha$.
Hence, for every $y\in X$
\[
g(y)\ge\tfrac{1}{2}\|y-x\|^{2}+\langle y,w^{*}\rangle+\alpha\ge-\tfrac{1}{2}\|w^{*}\|^{2}+\langle x,x^{*}\rangle+\alpha.
\]

Apply the Ekeland variational principle for $g\in\Gamma(X)$, $x\in\operatorname*{dom}g$,
to get $a_{\epsilon}$ a minimum point of $g+2\sqrt{\epsilon}\|\cdot-a_{\epsilon}\|\in\Gamma(X)$.
Equivalently, $0\in\partial(g+2\sqrt{\epsilon}\|\cdot-a_{\epsilon}\|)(a_{\epsilon})$.
Therefore
\[
0\in\partial(g+2\sqrt{\epsilon}\|\cdot-a_{\epsilon}\|)(a_{\epsilon})=\partial g(a_{\epsilon})+2\sqrt{\epsilon}U_{X^{*}}=\partial f(a_{\epsilon})+J_{X}(a_{\epsilon}-x)+2\sqrt{\epsilon}U_{X^{*}},
\]
where $U_{X^{*}}=\partial\|\cdot\|(0)$ denotes the closed unit ball
of $X^{*}$ and $J_{X}\coloneqq\partial(\tfrac{1}{2}\|\cdot\|^{2})$
is the du\-al\-i\-ty mapping, $J_{X}(x)=\{x^{*}\in X^{*}\mid\langle x,x^{*}\rangle=\|x\|^{2}=\|x^{*}\|^{2}\}$,
$x\in X$.

There exist $a_{\epsilon}^{*}\in\partial f(a_{\epsilon})$, $y_{\epsilon}^{*}\in J_{X}(a_{\epsilon}-x)$,
$u_{\epsilon}^{*}\in U_{X^{*}}$ such that $a_{\epsilon}^{*}=-y_{\epsilon}^{*}-2\sqrt{\epsilon}u_{\epsilon}^{*}$.

We have 
\[
\langle x-a_{\epsilon},a_{\epsilon}^{*}\rangle=\langle a_{\epsilon}-x,y_{\epsilon}^{*}\rangle-2\sqrt{\epsilon}\langle x-a_{\epsilon},u_{\epsilon}^{*}\rangle\ge\|x-a_{\epsilon}\|^{2}-2\sqrt{\epsilon}\|x-a_{\epsilon}\|\ge-\epsilon.
\]
By Theorem \ref{max-sd-sp} (viii) $\Rightarrow$ (i), $\partial f\in\mathfrak{M}(X)$.
\end{proof}

\strut

\section{Concluding remarks}

The results of this article suggest the following natural problems.
In view of the preceding characterizations, these problems may be
approached through the upper-envelope and low-upper-envelope identities
developed above.

\strut

\textbf{(P1)} Characterize those locally convex spaces $(X,\tau)$
for which $\partial f$ is maximal monotone for every $f\in\Gamma_{\tau}(X)$. 

\strut

\textbf{(P2)} Find intrinsic characterizations of those functions
$f\in\Gamma_{\tau}(X)$ for which $\partial f$ is maximal monotone.
Ideally, such characterizations should parallel the known envelope
characterization of maximal cyclic monotonicity: namely, for $f\in\Gamma_{\tau}(X)$,
$\partial f$ is maximal cyclically monotone if and only if $f=f^{\cup}$
and $f^{*}=f^{\circledast}$ (see \cite[Theorem 36]{gm}). 

\bibliographystyle{plain}

\end{document}